%% file: main.tex
\documentclass[11pt]{article}

\usepackage[ruled,vlined,linesnumbered,hangingcomment]{algorithm2e}
\usepackage[colorlinks=true,linkcolor=blue,citecolor=red,urlcolor=blue]{hyperref}
\usepackage[dvipsnames]{xcolor}
\usepackage{amsfonts,amssymb}
\usepackage{comment}
\usepackage{titlesec}
\usepackage{tikz}
\usepackage{bbding}
\usepackage{pifont}
\usepackage{titletoc}
\usepackage{hhline}
\usepackage{amsmath}
\usepackage{listings}
\usepackage{verbatim}
\usepackage{makecell}
\usepackage{tabu}
\usepackage{graphicx}
\usepackage{enumitem}
\setlist[itemize]{leftmargin=*}
\setlist[enumerate]{leftmargin=*}
\usepackage{mathrsfs}
\usepackage{setspace}
\usepackage{bbm}
\usepackage{mathtools}
\usepackage{titletoc}
\usepackage{geometry} 
\usepackage{threeparttable}
\usepackage{url} 
\usepackage{algpseudocode}

\usepackage{listings}
\input{preamble}

\usepackage[T1]{fontenc}
\usepackage[utf8]{inputenc}
 
\usepackage{amsthm}
\usepackage{xcolor, tikz}  
 \newtheorem{pro}{Proposition}[section]

\newtheorem{lem}{Lemma}[section]

\newtheorem{rem}{Remark}[section]

\counterwithin{theorem}{section}
\counterwithin{assumption}{section}

\usepackage{bm}
 
\DeclareMathOperator{\sign}{sign}

\DeclareMathOperator{\Bern}{\mathrm{Bernoulli}}

\DeclareMathOperator{\mle}{MLE}

\newif\ifjrnotes
\jrnotestrue   

\ifjrnotes
  \newcommand{\jrnote}[1]{%
    {\color{cyan}[JR: #1]}%
  }
\else
  \newcommand{\jrnote}[1]{}
\fi

\title{Minimax Optimal Estimator and Improved Error Rate for the MLE in Logistic Regression with Gaussian Design}  

\author{Junren Chen\thanks{Department of Statistics, Columbia University. (\texttt{jc6315@columbia.edu})}\and Arya Mazumdar\thanks{Hal{\i}c{\i}o\u{g}lu Data Science Institute, UC San Diego. (\texttt{arya@ucsd.edu})} 
}
\date{\today}

\begin{document}
\maketitle

\begin{abstract}
We study finite-sample parameter estimation in logistic regression with Gaussian design, where the goal is to estimate $\btheta^*\in \mathbb{R}^d$ with $R=\|\btheta^*\|_2\ge 1$ from i.i.d. samples $\{(\bx_i,y_i)\}_{i=1}^n,$ $\bx_i \sim N(\bm{0},\bI_d)$, $y_i\mid \bx_i \sim \Bern((1+\exp(-\bx_i^\top \btheta^*))^{-1})$. In this paper, we provide the first minimax optimal estimator, and 
 improve on the best known finite-sample error rate for the maximum likelihood estimator (MLE). 
 These two accomplishments are due to a minimax optimal estimator for
   the parameter norm $R$. First, we establish the minimax lower bound  $\Omega(\sqrt{R^3/n})$ for norm estimation. We then improve the best known norm estimation error rate of the MLE, i.e., $O(\sqrt{R^3d/n})$ from \cite{chardon2024finite}, to $\tilde{O}(\sqrt{R^3/n}+R^2d/n)$. The additional term, $R^2d/n$, appears to be the intrinsic bias of the MLE, as evidenced by the high-dimensional asymptotic theory \cite{zhao2022asymptotic} and numerical examples. We show that, however, this additional term is not information-theoretically necessary. To this end, we construct  an efficient debiased norm estimator that achieves the error rate $O(\sqrt{R^3/n})$   and is therefore minimax optimal. Combining this with the optimal direction estimator given by the MLE, we establish the minimax optimal rate $\Theta(\sqrt{Rd/n}+\sqrt{R^3/n})$ for estimating $\btheta^*$, as well as the improved finite-sample error rate $\tilde{O}(\sqrt{Rd/n}+\sqrt{R^3/n}+R^2d/n)$ for the MLE. Numerical experiments demonstrate that the proposed minimax optimal estimators outperform the MLE.
\end{abstract} 

\section{Introduction}\label{sec:intro}
Logistic regression is arguably one of the most basic models in statistics and machine learning. Conditioning on the covariates $\{\bx_i\}_{i=1}^n$, for some underlying parameter $\btheta^*\in \mathbb{R}^d$, the binary responses are independent and follow  
\begin{align}
y_i \mid\bx_i\sim \Bern(s(\bx_i^\top\btheta^*)),\quad\textrm{where ~}  s(a) = (1+\exp(-a))^{-1}, \label{logitsticmodel}
\end{align}
$i\in [n]:=\{1,2,\cdots,n\}$. This paper focuses on the finite-sample parameter estimation problem in logistic regression. As with recent works \cite{hsu2024sample,matsumoto2025learning,chardon2024finite,kuchelmeister2024finite,chen2026finite}, we study Gaussian design where the covariates $\bx_1,\cdots,\bx_n$ are i.i.d. $N(\bm{0},\bI_d)$ vectors.  Throughout the paper, we write $\bv^*:= \btheta^*/\|\btheta^*\|_2$, $R:=\|\btheta^*\|_2$, and assume $R\ge 1$. We use $C,C_i,c,c_i$ to denote positive universal constants whose values can vary from line to line, write  $T_1 \gtrsim T_2$ (or $T_1=\Omega(T_2)$) if $T_1\ge cT_2$, and write  $T_1\lesssim T_2$ (or $T_1=O(T_2)$) if $T_1\le CT_2$.

One of the most canonical estimators is the maximum likelihood estimator (MLE):\footnote{Note that \cite{hsu2024sample,matsumoto2025learning,chardon2024finite,kuchelmeister2024finite} adopt $\pm1$-valued responses $\tilde{y}_i=2y_i-1$, under which the logistic loss reads $L_n(\btheta)=\frac{1}{n}\sum_{i=1}^n\log(1+\exp(-\tilde{y}_i\bx_i^\top\btheta))$. This is a mathematically equivalent formulation.}
\begin{align} \label{mlelogi}
    \hat{\btheta}_{\mle}:=\textrm{arg}\min_{\btheta\in \mathbb{R}^d}\,\frac{1}{n}\sum_{i=1}^n\big(\log(\exp(\bx_i^\top\btheta)+1)-y_i\bx_i^\top\btheta\big):=\textrm{arg}\min_{\btheta\in \mathbb{R}^d}\,L_n(\btheta),
\end{align}
which exists only when the data is non-separable. 
It was shown by Chardon, Lerasle, and Mourtada \cite[Theorem 1 \& Equation (19)]{chardon2024finite} that under $n\gtrsim Rd$, with high probability (w.h.p.) the MLE exists and yields the direction estimate $\hat{\bv}_{\mle}:= \hat{\btheta}_{\mle}/\| \hat{\btheta}_{\mle}\|_2$ and norm estimate $\hat{R}_{\mle}:=\| \hat{\btheta}_{\mle}\|_2$ satisfying 
\begin{align}\label{chardonrate}
    \| \hat{\bv}_{\rm MLE} -\bv^*\|_2\lesssim \sqrt{\frac{d}{nR}}\qquad\textrm{and}\qquad |\hat{R}_{\mle}-R| \lesssim \sqrt{\frac{R^3d}{n}}. 
\end{align}  
This directly implies the full parameter estimation error rate 
\begin{align}
    \label{chardonratefull}
    \|\hat{\btheta}_{\mle}-\btheta^*\|_2 \lesssim \sqrt{\frac{R^3d}{n}}.
\end{align}
To the best of our knowledge, (\ref{chardonrate})--(\ref{chardonratefull}) are the best known finite-sample parameter estimation guarantees for the MLE; in fact, while some of the earlier results, such as \cite{ostrovskii2021finite,chinot2020robust}, also capture the scaling $O(\sqrt{d/n})$, they exhibit worse dependence on $R$. See \cite[Section 1.2]{chardon2024finite} and references therein. Moreover, (\ref{chardonrate}) suggests that direction estimation becomes easier while norm estimation becomes harder as the parameter norm $R$ increases. Comparable results (with extra logarithmic factors) have also been established in \cite[Theorem 2.1.1]{kuchelmeister2024finite} in a probit model.

More recently, \cite{chen2026finite} analyzed the finite-sample   performance of gradient descent on logistic loss  
\(\hat{\btheta}_{t+1} = \hat{\btheta}_t - \eta \nabla L_n(\hat{\btheta}_t),~t \ge 0\)
but only established estimation guarantees inferior to  (\ref{chardonrate}). More relevant to this paper is \cite[Theorem 3]{chen2026finite} which shows that some efficient estimator, say, $\tilde{\btheta}$, achieves 
\begin{align}\label{ssrate}
    \|\tilde{\btheta}-\btheta^*\|_2=\tilde{O}\bigg(\sqrt{\frac{Rd}{n}}+\sqrt{\frac{R^6}{n}}+\frac{R^2d}{n}\bigg)
\end{align}
under $n=\tilde{\Omega}(Rd+R^4)$. Note that we use $\tilde{O}(\cdot)$ and $\tilde{\Omega}(\cdot)$ to hide the logarithmic factors in $n,d,R$ in $O(\cdot)$ and $\Omega(\cdot)$, respectively. Therefore, (\ref{ssrate}) improves on (\ref{chardonratefull}) whenever $d=\tilde{\Omega}(R^3)$, though one cannot claim that $\tilde{\btheta}$ is a better estimator than the MLE since  (\ref{chardonratefull}) may not be sharp. Indeed, concerning full parameter estimation in such a foundational model, the following questions remain open:
\begin{itemize}
\item Are the MLE guarantees (\ref{chardonrate})--(\ref{chardonratefull}) tight or improvable?  

     \item What is the minimax-optimal Euclidean norm error rate in the estimation of $\btheta^*$?

    \item Is there a  minimax optimal estimator that can be computed from the data efficiently?  
\end{itemize}

On the other hand, the minimax optimal error rate and efficient algorithms for direction estimation are known. Here and hereafter, we shall restrict our attention to $n\gtrsim Rd$.\footnote{Substantial parts of our developments analyze or leverage the MLE, while \cite[Theorem 2]{chardon2024finite} showed that MLE does not exist w.h.p. if $n\lesssim Rd$. Note that the asymptotic, precise phase transition of the existence of the MLE was established in \cite{candes2020phase}.} In this regime, for any fixed $R\ge 1$, \cite[Theorem 1]{hsu2024sample} establishes the minimax lower bound 
\begin{align}
    \inf_{\tilde{\bv}((\bx_i,y_i)_{i=1}^n)} \sup_{\bv^* \in \mathbb{S}^{d-1}} \mathbb{E}_{R\bv^*}\|\tilde{\bv}-\bv^*\|_2 \gtrsim \sqrt{\frac{d}{nR}},
\end{align}
where the expectation is with respect to $(\bx_i,y_i)_{i=1}^n$ generated from $\btheta^*=R\bv^*$, and $\tilde{\bv}((\bx_i,y_i)_{i=1}^n)$ denotes an estimator constructed from $(\bx_i,y_i)_{i=1}^n$. Moreover, this lower bound is attained by gradient descent on the ReLU loss \cite{matsumoto2025learning}, up to logarithmic factors, and is also attained by the MLE up to a universal constant \cite{chardon2024finite}; see (\ref{chardonrate}).

As such, the key to answer the above open questions is to study the estimation of $R$, which is referred to as \emph{parameter strength}, \emph{signal-to-noise ratio}, or \emph{inverse temperature} in the studies of logistic regression (e.g., \cite{hsu2024sample,chardon2024finite}), but will simply be termed (parameter) norm in the present paper.  Our main contributions are as follows: 
\begin{itemize}
    \item In Theorem \ref{thm:minmaxlower}, we establish the minimax lower bound $\Omega(\sqrt{R^3/n})$ for norm estimation; 
    \item In Theorem \ref{thm:upper}, we improve the norm estimation guarantee for the MLE to 
    \begin{align}
        |\hat{R}_{\mle}-R|=\tilde{O}\bigg(\sqrt{\frac{R^3}{n}}+\frac{R^2d}{n}\bigg);\label{mleimprovenorm}
    \end{align}
    \item We develop a debiased norm estimator in Algorithm \ref{alg:debiased-norm-estimate}, and show in Theorem \ref{thm:minimaxupper} that it achieves $O(\sqrt{R^3/n})$ and is therefore minimax optimal. 
\end{itemize}

The first and third bullet points imply that the minimax optimal rate for norm estimation is at the order of $\sqrt{R^3/n}$. In contrast,  our improved norm estimation error rate (\ref{mleimprovenorm}) for the MLE is suboptimal due to an additional {\it dimension-dependent} term $R^2d/n$. We conjecture that this term is intrinsic to the MLE rather than a proof artifact, along with some evidence from the asymptotic theory (Proposition \ref{prop:asymptotic-radial-bias}) and numerical simulations (Figure \ref{fig:normbias}).

Combining these results with known results on direction estimation, we immediately come to the following answers to the above open questions (concerning estimation of $\btheta^*$): 
\begin{itemize}
    \item The minimax optimal rate is at the order of $\sqrt{Rd/n}+\sqrt{R^3/n}$ and can be attained by an efficient algorithm; see Theorems \ref{thm:fullminimax}, \ref{thm:fullminimaxupper}.  
    
    \item The error rate of the MLE is improved to 
    \(\|\hat{\btheta}_{\mle}-\btheta^*\|_2 =\tilde{O}(\sqrt{Rd/n}+\sqrt{R^3/n}+R^2d/n)\), see Theorem \ref{thm:fullmle}.     
\end{itemize}

The derivations of our results are technical. The lower bound in Theorem \ref{thm:minmaxlower} is established by Le Cam's inequality.  The improved error rate of the MLE in Theorem \ref{thm:upper} is built upon Lemma \ref{lem:chardonbound} from \cite{chardon2024finite}, with the key components being a deterministic bound from the first-order optimality and a localized covering argument. For the minimax upper bound in Theorem \ref{thm:minimaxupper}, the development of the debiased norm estimator itself is highly nontrivial.

The remainder of this paper is organized as follows. In Section \ref{sec:main}, we introduce our main results on the parameter norm estimation. Combining with the optimal direction estimation results, our full parameter estimation results are presented in Section \ref{sec:fullestimate}. We provide numerical examples in Section \ref{sec:numerics} to demonstrate our theoretical findings. Based on the proof sketches following Theorems \ref{thm:upper} and \ref{thm:minimaxupper}, further mathematical details are provided in Sections \ref{sec:proofupper} and \ref{sec:proofminimaxupper}. We close the paper with concluding remarks in Section \ref{sec:conclude}. The elementary Gaussian moment bounds, some standard mathematical details in the proofs of Theorems \ref{thm:upper}--\ref{thm:minimaxupper}, the complete proof for the secondary result Proposition \ref{prop:asymptotic-radial-bias}, and the technical lemmas are relegated to the appendices.

\section{Parameter Norm Estimation}\label{sec:main}
We start with a precise formulation of our model.  
\begin{assumption}[Logistic regression with Gaussian design] 
    \label{assump1} For some unknown parameter $\btheta^*\in \mathbb{R}^d$ with norm $R:=\|\btheta^*\|_2\ge 1$ and direction $\bv^*=\btheta^*/\|\btheta^*\|_2$, the observations $(\bx_1,y_1),(\bx_2,y_2),\cdots,(\bx_n,y_n)$ are i.i.d. with $\bx_i\sim N(\bm{0},\bI_d)$, $y_i\mid \bx_i\sim \Bern(s(\bx_i^\top\btheta^*))$. 
\end{assumption}

Before proceeding, we shall introduce a recurring function.
For \(r\ge 0\), as with \cite[Lemma 15]{chen2026finite}, we define 
\(
    q(r):=\mathbb{E}_{g\sim N(0,1)}[g s(rg)].\) 
   This function arises in   the population gradient of the logistic loss; see Lemma \ref{lem:population-gradient}. We also note that $q(r)=O(1)$ in light of 
\(
    \sup_{r\geq 0}|q(r)| \le \mathbb{E}_{g\sim N(0,1)}|g| = \sqrt{2/\pi}\), and that $q'(r) = \mathbb{E}_{g\sim N(0,1)}[g^2 s'(rg)]$ satisfies $q'(r)\asymp (1+r)^{-3}$ over $r\ge 0$; see Lemma \ref{lem:qprime}. Here and hereafter, we write $T_1\asymp T_2$ (or $T_1=\Theta(T_2)$) if $cT_2<T_1<CT_2$ for some universal constants $C>c>0$.

\subsection{Minimax Lower Bound}

For any $\btheta\in \mathbb{R}^d$, we use $\mathbb{E}_{\btheta}$ to denote the expectation with respect to $\{(\bx_i,y_i)\}_{i=1}^n$ generated from the underlying parameter $\btheta$. In the following, we provide an {\it algorithm-independent} minimax lower bound on the parameter norm estimation.

\begin{theorem}[Minimax Lower Bound for Norm Estimation]
\label{thm:minmaxlower}
Under Assumption \ref{assump1}, there exist universal constants
$C>c>0$ such that the following holds.
Fix $R_0\ge 1$ and $\bv^*\in\mathbb{S}^{d-1}$.
If $n\ge CR_0$, then   
\begin{align}
\inf_{\tilde R
=
\tilde R((\bx_i,y_i)_{i=1}^n,\bv^*)}
\sup_{R\in[R_0,2R_0]}
\mathbb{E}_{R\bv^*}
\big|\tilde R-R\big|
\ge
c\sqrt{R_0^3/n}.
\label{normlower}
\end{align}
\end{theorem}
\begin{proof}
Let $P_R$ denote the distribution of one observation $(\bx,y)$
under $\btheta^*=R\bv^*$, i.e., $\bx\sim N(\bm{0},\bI_d)$, $y\mid \bx \sim \Bern(s(\bx^\top\btheta^*))$.
Set
\(
R_1=R_0,~
R_2=R_0+h,
~h=\kappa\sqrt{ R_0^3/n},\)
where $\kappa>0$ is a sufficiently small universal constant.
Under $n\ge CR_0$, we have $h\le R_0$ and hence
$R_1,R_2\in[R_0,2R_0]$. Let $Z=\bx^\top\bv^*\sim N(0,1)$. We write $D_{\rm KL}(P\|Q)$ and $\|P-Q\|_{\rm TV}$ for the
Kullback--Leibler divergence and total variation distance,
respectively; see Lemma \ref{lem:pinsker} for their definitions. 
Since the distribution of $\bx$ does not depend on $R$,
tensorization and the chain rule of the KL divergence give
\begin{align}
D_{\rm KL}
\left(
P_{R_1}^{\otimes n}
\,\middle\|\,
P_{R_2}^{\otimes n}
\right)
=
n\mathbb{E}_Z
D_{\rm KL}
\left(
\Bern(s(R_1Z))
\,\middle\|\,
\Bern(s(R_2Z))
\right).
\label{eq:norm-lower-kl-chain}
\end{align}
Let $\rho(u)=\log(1+e^u)$ so that
$\rho'(u)=s(u)$ and $\rho''(u)=s'(u)$.
For any $a,b\in\mathbb{R}$, we have 
\[
D_{\rm KL}\big(\Bern(s(a))\|\Bern(s(b))\big)
=
\rho(b)-\rho(a)-\rho'(a)(b-a).
\]
Taylor's theorem, together with the fact that $s'$ is even and
nonincreasing on $[0,\infty)$,  yields
\[
D_{\rm KL}
\left(
\Bern(s(R_1Z))
\,\middle\|\,
\Bern(s(R_2Z))
\right)
\le
\frac{h^2Z^2}{2}s'(R_0|Z|).
\]
Substituting this into \eqref{eq:norm-lower-kl-chain} and invoking
\eqref{T4bound1} with $p=2$, we obtain
\[
D_{\rm KL}
\left(
P_{R_1}^{\otimes n}
\,\middle\|\,
P_{R_2}^{\otimes n}
\right)
\le
C\frac{nh^2}{R_0^3}
=
C\kappa^2.
\]
Choosing $\kappa$ sufficiently small, Pinsker's inequality (Lemma \ref{lem:pinsker}) gives \(
\big\|
P_{R_1}^{\otimes n}
-
P_{R_2}^{\otimes n}
\big\|_{\rm TV}
\le \frac12. 
\)
Le Cam's inequality (Lemma \ref{lem:lecam-two-point}) then implies
\begin{align*}
\inf_{\tilde R
=
\tilde R((\bx_i,y_i)_{i=1}^n,\bv^*)}
\sup_{R\in[R_0,2R_0]}
\mathbb{E}_{R\bv^*}|\tilde R-R|
&\ge
\inf_{\tilde R}
\max_{j\in\{1,2\}}
\mathbb{E}_{R_j\bv^*}|\tilde R-R_j|\ge
\frac{h}{4}
\left(
1-
\left\|
P_{R_1}^{\otimes n}
-
P_{R_2}^{\otimes n}
\right\|_{\rm TV}
\right)\ge
\frac{h}{8}.
\end{align*}
The result follows from the definition of $h$.
\end{proof} 
\begin{rem}
    Note that Theorem \ref{thm:minmaxlower} is an oracle lower bound that remains valid when the true direction $\bv^*$ is revealed to the estimator. Consequently, it holds when $\bv^*$ is unknown.
\end{rem}

\subsection{Improved Guarantee for the MLE}
\begin{theorem}\label{thm:upper}
Under Assumption \ref{assump1}, there exist universal constants $C>c>0$ such that, if $n\ge CRd$, then with probability at least $1-Ce^{-cd}-Cn^{-1}$, $\hat{\btheta}_{\mle}=\hat{R}_{\mle}\hat{\bv}_{\mle}$ in (\ref{mlelogi})  exists and satisfies 
\begin{align}\label{thm1norm}
      |\hat{R}_{\mle}-R| \le C\bigg(\sqrt{\frac{R^3\log(nR)}{n}}+\frac{R^2d\log(nR)}{n}\bigg).
\end{align}
\end{theorem}
 
Ignoring logarithmic factors, our norm estimation rate for the MLE reads $\sqrt{R^3/n}+R^2d/n$ and strictly improves on $O(\sqrt{R^3d/n})$ in (\ref{chardonrate}) due to \cite{chardon2024finite}. The complete proof is rather involved, so we only outline the proof architecture here and leave the details to Section \ref{sec:proofupper}. 
\begin{proof}[Proof Sketch for Theorem \ref{thm:upper}]
    Unlike \cite{chardon2024finite} that essentially used $|\hat{R}_{\mle}-R|\le \|\hat{\btheta}_{\mle}-\btheta^*\|_2$, 
we start with a finer deterministic bound derived from the first-order optimality (Lemma \ref{lem:radial-score}):
\begin{align}
    |\hat R_{\mle}-R|
    \lesssim R^3
    \big\{
        \|\hat \bv_{\mle}-\bv^*\|_2^2
        +
        |S_n(\hat R_{\mle},\hat \bv_{\mle})|
    \big\},\label{boundstartpoint}
\end{align}
where $S_n(r,\bv)=\langle \nabla L_n(r\bv)-\mathbb{E}[\nabla L_n(r\bv)],\bv\rangle$, and $\nabla L_n(\btheta)=\frac{1}{n}\sum_{i=1}^n(s(\bx_i^\top\btheta)-y_i)\bx_i$ is the gradient of the logistic loss. By the optimal direction estimation error rate in (\ref{chardonrate}), $R^3\|\hat{\bv}_{\mle}-\bv^*\|_2^2\lesssim R^2d/n$ is within the final error. All that remains is to control $|S_n(\hat{R}_{\mle},\hat{\bv}_{\mle})|$, where the main difficulty lies in   the   dependence of $(\hat{R}_{\mle},\hat{\bv}_{\mle})$ on the data, rendering the term a non-independent sum. Built upon the known bounds in (\ref{chardonrate}) that imply $\hat{R}_{\mle}\in [R/2,2R]$ under $n\gtrsim Rd$, we simply take the supremum and study the worst-case bound:
\[|S_n(\hat{R}_{\mle},\hat{\bv}_{\mle})| \le \sup_{\|\bv-\bv^*\|_2\lesssim (\frac{d}{nR})^{1/2}}\sup_{ r\in[R/2,2R]}\,|S_n(r,\bv)|\le \sup_{r}|S_n(r,\bv^*)|+\sup_{r,\bv}|S_n(r,\bv)-S_n(r,\bv^*)|.\]
We then invoke Bernstein's inequality (cf. Lemma \ref{lem:bernstein210}) and   covering argument to control both processes. 
Specifically,
Lemma \ref{lem:true-direction-process} controls the first process as $\tilde{O}( \frac{1}{\sqrt{nR^3}})$, yielding the final term $\sqrt{R^3/n}$ when multiplied by $R^3$ as in (\ref{boundstartpoint}); note that this is a one-dimensional process and the covering over $r\in [R/2,2R]$
 only loses logarithmic factors (compared to the non-uniform bound).  For the second term, while a covering argument over $\bv$ loses some dependence on the dimension,  it remains sufficient because it is
applied to a localized set $\calV:=\{\bv:\|\bv-\bv^*\|_2\lesssim\sqrt{d/(nR)}\}$ rather than the entire sphere.   
The key observation is that, for a fixed
$(r,\bv)$, Bernstein's inequality is capable of yielding a non-uniform bound on $|S_n(r,\bv)-S_n(r,\bv^*)|$ proportional to $\|\bv-\bv^*\|_2$, precisely $\tilde{O}(\|\bv-\bv^*\|_2/\sqrt{nR})$, which reads $\tilde{O}(\frac{\sqrt{d}}{nR})$ over $\calV$; see (\ref{fixedrvbound}). A covering argument degrades this by a factor of $\sqrt{d}$ to $\tilde{O}(\frac{d}{nR})$, yielding the final bound $\tilde{O}(R^2d/n)$ when further amplified by $R^3$ as in (\ref{boundstartpoint}). This is unproblematic, up to logarithmic factors, since the term $R^2d/n$ already appears from $R^3\|\hat{\bv}_{\mle}-\bv^*\|_2^2$ and cannot be avoided by the current proof mechanism. 
\end{proof}

\subsubsection{Is $R^2d/n$ intrinsic to the MLE?}
Up to logarithmic factors, our Theorems \ref{thm:upper}, \ref{thm:minmaxlower} show that MLE achieves finite-sample norm estimation rate $\sqrt{R^3/n}+R^2d/n$, and that $\sqrt{R^3/n}$ cannot be avoided by any estimator constructed from $\{(\bx_i,y_i)\}_{i=1}^n$. 
As such, to understand the canonical estimator of the MLE, the remaining piece concerns the tightness of the dimension-dependent term $R^2d/n$: is this term fundamental or simply a proof artifact? 
We conjecture that the term \(R^2d/n\) reflects \emph{an intrinsic
positive radial bias of the MLE}, rather than a proof artifact.
In the following, we provide rigorous evidence from the high-dimensional asymptotic theory   \cite{zhao2022asymptotic,sur2019modern}.

\begin{pro}
\label{prop:asymptotic-radial-bias}Fix $R\ge1$, and consider a sequence of logistic models with
$\btheta_n^*\in\mathbb R^{d_n}$ satisfying
$\|\btheta_n^*\|_2=R$ and $d_n/n\to\delta>0$ (in the regime where the MLE exists asymptotically \cite{candes2020phase}). Then there exists a deterministic quantity \(r_{\rm MLE}(\delta,R)\) such that
$
    \hat{R}_{\mle}
    \xrightarrow{\rm a.s.}
    r_{\rm MLE}(\delta,R).$ 
Moreover, for every fixed \(R\),
\[
    r_{\rm MLE}(\delta,R)-R
    =
    \frac{\delta}{2R q'(R)}
    +O_R(\delta^2), \quad \textrm{as}~
     \delta\to 0^+,
\]
where
   $ q'(R)=\mathbb{E}_{g\sim N(0,1)}\!\left[g^2s'(Rg)\right].$ Here, the implicit constant in $O_R(\cdot)$ may depend on the fixed
value of $R$.
\end{pro}

This result follows from some elementary calculations---as $r_{\mle}(\delta,R)$ has been defined in \cite{zhao2022asymptotic} implicitly through a system of equations, all we need here is to expand $r_{\mle}(\delta,R)$ near $\delta=0$. We position this as our side result and therefore postpone the proof to Appendix \ref{sec:small-aspect-ratio-bias}. In the following, we explain why Proposition \ref{prop:asymptotic-radial-bias} suggests that $R^2d/n$ may be intrinsic to the MLE. 

\begin{rem}\label{rem:mlenormopen}
    In the high-dimensional asymptotic regime $d/n\to \delta$, Proposition \ref{prop:asymptotic-radial-bias} asserts that $r_{\mle}(\delta,R)$ satisfies
    \[\lim_{\delta\to 0^+}\frac{r_{\rm MLE}(\delta,R)-R}{\delta} = \frac{1}{2Rq'(R)}\]
    for any fixed $R\ge 1$. Since $1/Rq'(R)\asymp R^2$ and $\delta$ is the limit of the aspect ratio $\frac{d}{n}$, it suggests that the MLE overestimates the true norm $R$ by a bias of order $R^2d/n$. 
    While Proposition \ref{prop:asymptotic-radial-bias} holds only in specific high-dimensional asymptotics and under fixed $R$, the numerical examples in Section \ref{sec:numerics} suggest that $d/(2Rq'(R)n)$  remains an accurate prediction for the average behavior of $\hat{R}_{\mle}-R$ in finite samples, where $n,d$ may not be proportional and $R$ may increase with $n$.  
    However, we shall emphasize the following: while the above Proposition \ref{prop:asymptotic-radial-bias} and the forthcoming numerical examples provide evidence on the tightness of $R^2d/n$ for the MLE, it remains an open question to establish that $\hat{R}_{\mle}-R \gtrsim  R^2d/n$ in finite samples.
\end{rem}

\subsection{Minimax Optimal Estimator}\label{sec:minimaxnorm}
In this section, we construct an efficient algorithm that achieves the minimax lower bound in Theorem \ref{thm:minmaxlower}. 
We devise a debiased norm estimator for this purpose in Algorithm \ref{alg:debiased-norm-estimate}, where we let \(\Pi_{[a,b]}(x):=\min\{\max\{x,a\},b\},~
    [x]_+:=\max\{x,0\}.\)

 \begin{algorithm}[ht!]
\caption{A Debiased Norm Estimator}
\label{alg:debiased-norm-estimate}
\begin{algorithmic}[1]

\Require  
\(\{(\bx_i,y_i)\}_{i=1}^{n}
\subset \mathbb{R}^{d}\times\{0,1\}\) and the positive integers $(m_1,m_2,m_3)$ such that  $m_1+m_2+m_3=n$. Let $\calI_1:=[m_1]$, $\calI_2=[m_1+m_2]\setminus[m_1]$, $\calI_3= [n]\setminus [m_1+m_2]$

\State  \textbf{Step 1: Direction estimation via MLE using $\{(\bx_i,y_i)\}_{i\in\calI_1}$}

\Statex Compute 
\begin{align}
    \breve{\btheta}
\in
\mathop{\arg\min}_{\btheta\in\mathbb{R}^{d}}
\frac{1}{m_1}
\sum_{i=1}^{m_1}
\left[
\log\big(1+
e^{\bx_i^{\top}\btheta}\big)
-y_i\bx_i^{\top} \btheta
\right]. \label{smallmle}
\end{align}

\Statex Let its direction be $\breve{\bv}
:=
 \breve{\btheta}/
\|\breve{\btheta}\|_2$ and norm be $\breve{R}:=\|\breve{\btheta}\|_2$. 

\State \textbf{Step 2: Norm estimate via ``misspecified'' 1-dimensional MLE using $\{(\bx_i,y_i)\}_{i\in\calI_2}$}

\Statex Compute
\begin{align}\label{1dmle}
    \hat{r}
\in
\mathop{\arg\min}_{r\in [0,4\breve{R}]}
\,\frac{1}{m_2}
\sum_{i\in\calI_2} 
\left[
\log\!\left(
1+e^{r\bx_i^{\top}\breve{\bv}}
\right)
-y_i r\bx_i^{\top}\breve{\bv}
\right].
\end{align}

\State \textbf{Step 3: Debiasing the norm estimate using $\{(\bx_i,y_i)\}_{i\in\calI_3}$} 

\Statex Let $\bP_{\breve{\bv}}^\perp =\bI_d- \breve{\bv}\breve{\bv}^\top$. For every $i\in\calI_3$, define
\begin{align} \label{uiconstruct}
    \bw_i
:=\bP_{\breve{\bv}}^\perp\big[
y_i-s(
\hat{r}\,  
\bx_i^{\top}\breve{\bv}
)
\big]
\bx_i. 
\end{align}

\Statex Let \(m_3:=|\mathcal{I}_3|\), and compute
\begin{align} \label{hatEconstruct}
    \hat{B}
:=
\frac{1}{m_3(m_3-1)}
\sum_{\substack{i,j\in\mathcal{I}_3\\i\neq j}}
\left\langle
\bw_i,\bw_j
\right\rangle. 
\end{align}

\Statex Form the corrected estimate $
    \hat{S}
:=
q(\hat{r})^2+\hat{B}.$ Define
\begin{align} \label{hatqdefined}
    \hat{q}
:=
\Pi_{
[\,q(\hat{r}/2),\,q(2\hat{r})\,]
}
\Big(
\sqrt{\big[\hat{S}\big]_+}
\Big).  
\end{align}

\Statex Return the debiased norm estimate  $\hat{R}_{\mathrm{DB}}
:=
q^{-1}(\hat{q}).$ 

\Ensure \(\hat{R}_{\mathrm{DB}}\).

\end{algorithmic}
\end{algorithm}

We illustrate some key ideas behind Algorithm \ref{alg:debiased-norm-estimate}. While we believe that the estimator is practically relevant 
(as we shall see in Section \ref{sec:numerics}), its main aim is to determine the minimax optimal rate for norm estimation. As such, we split the samples into several blocks to render fresh samples for different stages. Our analysis and the algorithm itself rely heavily on this sample splitting trick.

Suppose $m_1=m_2=m_3=n/3$. To motivate the debiasing step, we first consider the norm estimate $\hat{r}$ from (\ref{1dmle}). Note that Steps 1--2 can be viewed as a sample splitting counterpart of the MLE---Step 1 computes MLE to obtain a direction estimate, while Step 2 treats this direction estimate as the exact underlying direction $\bv^*$ and invokes a 1-dimensional MLE to yield a norm estimate.\footnote{As another angle, if Step 2 reuses the data from Step 1, then $\hat{r}\breve{\bv}$ is exactly the MLE on $\{(\bx_i,y_i)\}_{i\in\calI_1}$.}

However, in general $\breve{\bv}$ does not equal $\bv^*$, and therefore (\ref{1dmle}) is a misspecified 1-dimensional MLE that is biased in population. This is characterized by the following lemma. 

\begin{lem}[Population bias of the one-dimensional MLE]
\label{lem:population-bias-1dmle}
Condition on $\breve{\bv}$ and suppose
$
\alpha:=\langle \breve{\bv},\bv^*\rangle\in (0,1].$ 
Then the  population loss associated with \eqref{1dmle},
i.e., $
\mathcal L_{\breve{\bv}}(r)
:=
\mathbb{E}[
\log(1+e^{r\bx_i^\top\breve{\bv}})
-y_ir\bx_i^\top\breve{\bv}], 
$
has a unique minimizer $r^\circ$ characterized by
\begin{equation}
q(r^\circ)=\alpha q(R).
\label{eq:pseudo-true-radius}
\end{equation}
Moreover, if $R\geq1$ and
$
1-\alpha\leq \frac{c_0}{R^2}$ 
for a sufficiently small universal constant $c_0>0$, then
$r^\circ\in[3R/4,R]$ and
\begin{equation}
R-r^\circ
\asymp
R^3(1-\alpha)
=
\frac{R^3}{2}\|\breve{\bv}-\bv^*\|_2^2.
\label{eq:population-attenuation}
\end{equation}
\end{lem}

\begin{proof}
Let $g:=\bx_i^\top\breve{\bv}\sim N(0,1)$, differentiating the conditional
population loss gives
$
\mathcal L_{\breve{\bv}}'(r)
=
\mathbb{E}\!\left[g\{s(rg)-y_i\}\right]= q(r)- \mathbb{E}[gy_i].$ Moreover, by rotational invariance, 
\[\mathbb{E}[gy_i] = \mathbb{E}\big[\mathbb{E}[\bx_i^\top\breve{\bv}y_i\mid \bx_i]\big] = \mathbb{E}\big[s(R\bx_i^\top\bv^*)\bx_i^\top\breve{\bv}\big]= \langle\breve{\bv},\bv^*\rangle\mathbb{E}\big[s(R\bx_i^\top\bv^*)\bx_i^\top\bv^*\big]=\alpha \mathbb{E}[s(Rg)g]= \alpha q(R).\]
Therefore, we have
$
\mathcal L_{\breve{\bv}}'(r)=q(r)-\alpha q(R).$ 
Because
$
q'(r)=\mathbb{E}_{g\sim N(0,1)}[g^2s'(rg)]>0,$ 
the loss is strictly convex, and its unique minimizer satisfies
\eqref{eq:pseudo-true-radius}. Since $q$ is strictly increasing,
$r^\circ\leq R$, with strict inequality when $\alpha<1$. For the quantitative statement, we start with
$
q(R)-q(r^\circ)=(1-\alpha)q(R).$
The assumption $1-\alpha\leq c_0/R^2$, together with
$q(R)-q(3R/4)\gtrsim R^{-2}$, implies $r^\circ\geq3R/4$ when $c_0$
is sufficiently small. By the mean-value theorem, for some
$\tilde r\in[r^\circ,R]$,
\[
R-r^\circ
=
\frac{(1-\alpha)q(R)}{q'(\tilde r)}.
\]
Since $q(R)\asymp1$ and
$q'(\tilde r)\asymp R^{-3}$ uniformly over
$\tilde r\in[3R/4,R]$, we obtain
$
R-r^\circ\asymp R^3(1-\alpha)=\frac{R^3}{2}\|\breve{\bv}-\bv^*\|_2^2.$ 
\end{proof}

In light of Lemma~\ref{lem:population-bias-1dmle}, even if infinitely many samples are available, the minimizer of (\ref{1dmle}) approximates $r^\circ= q^{-1}(\alpha q(R))$ that underestimates the true norm $R$ by a gap of order $R^3\|\breve{\bv}-\bv^*\|_2^2$. Substituting $\|\breve{\bv}-\bv^*\|_2^2 =O(\frac{d}{nR})$ from (\ref{chardonrate}), (\ref{1dmle}) incurs a population bias  $R-r^\circ=O(\frac{R^2d}{n})$.  Moreover, a convex localization argument yields the high-probability bound $|\hat{r}-r^\circ|=\tilde{O}(\sqrt{R^3/n})$ (this is indeed an important piece of our analysis; see Lemma \ref{lem:profile-local}). Taken collectively, \[|\hat{r}-R|=\tilde{O}\bigg(\frac{R^2d}{n}+ \sqrt{\frac{R^3}{n}}\bigg),\]
which is comparable to our improved MLE norm estimation guarantee (\ref{thm1norm}) and does not match the lower bound $\Omega(\sqrt{R^3/n})$. 

As such, toward a minimax optimal norm estimate, we have to debias $\hat{r}$ and get rid of the population bias term $R^2d/n$. This is accomplished by Step 3 and is the main innovation of Algorithm \ref{alg:debiased-norm-estimate}. We now elaborate the idea. In view of Lemma~\ref{lem:population-bias-1dmle}, $q(\hat{r})^2$ can serve as a surrogate of $q(r^\circ)^2 = \alpha^2 q(R)^2$. Also, due to the monotonicity of $q$, an accurate estimate of $q(R)^2$ immediately yields a norm estimate by applying $q^{-1}$. Notice that $q(R)^2 = \alpha^2q(R)^2 + (1-\alpha^2) q(R)^2,$
all that remains is to construct a surrogate of $(1-\alpha^2)q(R)^2.$ In light of 
$\|\bP_{\breve{\bv}}^\perp\bv^*\|_2^2=1-(\breve{\bv}^\top\bv^*)^2=1-\alpha^2$,
this appeals to the technical constructions in (\ref{uiconstruct})--(\ref{hatEconstruct}), whose roles are justified by the following lemma. 
\begin{lem}\label{lem:EwEB}
    In the setting of Algorithm \ref{alg:debiased-norm-estimate}, by conditioning on $(\hat{r},\breve{\bv})$ and using the randomness of $\{(\bx_i,y_i)\}_{i\in\calI_3}$,  
    \[
        \mathbb{E}\bw_i =q(R)\bP^\perp_{\breve{\bv}}\bv^*,\qquad 
        \mathbb{E}\hat{B} = (1-\alpha^2) q(R)^2.
    \] 
\end{lem}
\begin{proof}
By Lemma \ref{lem:population-gradient}, 
we have \(\mathbb{E}\big[y_i-s(\hat{r}\bx_i^\top\breve{\bv})\big]\bx_i = q(R)\bv^* - q(\hat{r})\breve{\bv},\)
and therefore 
\[\mathbb{E}\bw_i = \bP^\perp_{\breve{\bv}}\big(q(R)\bv^*-q(\hat{r})\breve{\bv}\big) = q(R)\bP^\perp_{\breve{\bv}}\bv^*.\]
Since $\bx_i$'s are independent, we have
\[\mathbb{E}\hat{B} = \mathbb{E}\langle\bw_1,\bw_2\rangle = \langle \mathbb{E}\bw_1,\mathbb{E}\bw_2\rangle = q(R)^2\|\bP_{\breve{\bv}}^\perp\bv^*\|_2^2=q(R)^2\big(1-\langle\breve{\bv},\bv^*\rangle^2\big)=  (1-\alpha^2)q(R)^2.\]
The proof is complete. 
\end{proof}

We are ready to present our main result on the minimax optimality of Algorithm \ref{alg:debiased-norm-estimate}.

\begin{theorem}[Minimax optimality of Algorithm \ref{alg:debiased-norm-estimate}] 
    \label{thm:minimaxupper} Under Assumption \ref{assump1}, let $\hat{R}_{\rm DB}$  be the output of Algorithm \ref{alg:debiased-norm-estimate}. There exist universal constants $C>c>0$ such that, given any $t\ge 1$, if $\min\{m_1,m_2,m_3\}\ge cn,~n\ge C R(d+t)$, then with probability at least $1-C\exp(-ct)$,
    \begin{gather} \label{thm3:norm}
        |\hat{R}_{\rm DB}-R|\le C\sqrt{R^3t/n}\,.
    \end{gather}
\end{theorem}

Here, we only provide a sketch of the proof, while deferring details to Section \ref{sec:proofminimaxupper}.

\begin{proof}[Proof Sketch for Theorem \ref{thm:minimaxupper}]
Conditioning on $\breve{\bv},\breve R$ whose behavior  is characterized by Lemma \ref{lem:chardonbound} from \cite{chardon2024finite}, Lemma
\ref{lem:profile-local} locates $\hat{r}$ found by (\ref{1dmle}) and shows, via a convex localization argument, that the one-dimensional MLE in Step 2
satisfies
\(
|\hat r-r^\circ|
\lesssim
\sqrt{\frac{R^3t}{n}}.
\)  By some algebra, this gives
\(
|q(\hat r)^2-\alpha^2q(R)^2|
\lesssim
\sqrt{\frac{t}{nR^3}}.
\) Hence $q(\hat{r})$ gives a surrogate of $\alpha^2q(R)^2$.
Toward an estimate of  $q(R)^2$, we 
let $\beta^2:=1-\alpha^2$ and construct $\hat{B}$ in Step 3 as a surrogate of $\beta^2q(R)^2$. While its role has been justified in
Lemma \ref{lem:EwEB}, which establishes 
$\mathbb{E}\hat B=\beta^2q(R)^2$, we still need to control the concentration error $|\hat{B}-\mathbb{E}\hat B|$. This turns out to be the most technical part of the analysis, with the main components being multiple applications of Bernstein's inequality that are built upon a number of non-trivial Gaussian moment bounds (collected in Appendix \ref{app:integral}), as well as a Hanson-Wright inequality (cf. Lemma \ref{lem:hanson-wright}) for controlling a Gaussian quadratic form that arises from some technical algebraic manipulation. In particular, 
Lemma
\ref{lem:direct-U} establishes
\(
|\hat B-\beta^2q(R)^2|
\lesssim
\frac{\sqrt{dt}+t}{nR},
\)  which implies that $\hat S=q(\hat r)^2+\hat B$ accurately estimates $q(R)^2$ up to   overall error
\(
|\hat S-q(R)^2|
\lesssim
\sqrt{\frac{t}{nR^3}}
+
\frac{\sqrt{dt}+t}{nR}\lesssim\sqrt{\frac{t}{nR^3}}.
\) Finally, some elementary arguments, based on the properties of the function $q(r)$, yield the desired bound on $|\hat{R}_{\rm DB}-R|$.
\end{proof}
 \begin{rem}[Minimax optimality in expectation]\label{rem:Enormerror}
 Theorem \ref{thm:minimaxupper}   implies that $|\hat{R}_{\rm DB}-R|=O(\sqrt{R^3\log n/n})$ holds with probability at least $1-Cn^{-1}$, yet this bound exhibits an extra factor of $\sqrt{\log n}$ compared to the minimax lower bound in Theorem \ref{thm:minmaxlower}. Here, we point out that   a slightly modified version of $\hat{R}_{\rm DB}$ attains the minimax lower bound in expectation. 
Fix $R_0\ge1$, consider $R\in[R_0,2R_0]$ and define  
\(
    \hat R_{\rm DB}^{\rm tr}
    :=
    \Pi_{[R_0,2R_0]}(\hat R_{\rm DB}),\)
where an arbitrary value in $[R_0,2R_0]$ is assigned whenever Algorithm \ref{alg:debiased-norm-estimate} is not well-defined (e.g., when the MLE in (\ref{smallmle}) does not exist). Integrating the tail bound in Theorem \ref{thm:minimaxupper} gives 
\[
    \sup_{R\in[R_0,2R_0]}
    \mathbb{E}_{R\bv^*}
    \big|\hat R_{\rm DB}^{\rm tr}-R\big|
    \lesssim
    \sqrt{R_0^3/n}.
\]
Under $n\gtrsim R_0d$, this matches the lower bound in Theorem \ref{thm:minmaxlower} up to a universal constant.
 \end{rem}

\section{Full Parameter Estimation}\label{sec:fullestimate}
In this section, we present our results on the estimation of the full parameter $\btheta^*$ and put them into perspective. 

 \begin{theorem}[Minimax lower bound] \label{thm:fullminimax}
 Let $R_0\ge 1$, and let $\Theta_{R_0}:=\{\bu\in \mathbb{R}^d:R_0\le \|\bu\|_2 \le 2R_0\}$. If $n\ge C(R_0+\frac{d}{R_0})$, then any estimator $\tilde{\btheta}=\tilde{\btheta}((\bx_i,y_i)_{i=1}^n)$ for the true parameter $\btheta^*$ satisfies
    \begin{align}\label{jointlower}
        \sup_{\btheta^*\in \Theta_{R_0}}\mathbb{E}_{\btheta^*}\|\tilde{\btheta}-\btheta^*\|_2\ge c\bigg(\sqrt{\frac{R_0d}{n}}+\sqrt{\frac{R_0^3}{n}}\bigg).
    \end{align}
 \end{theorem}
 \begin{proof}
     Fix any estimator
$\tilde{\btheta}=\tilde{\btheta}((\bx_i,y_i)_{i=1}^n)$ and let
\(
    \tilde{\bv}
    :=
     \tilde{\btheta}/\|\tilde{\btheta}\|_2,\) under the convention $\bm{0}/0=\be_1$.
  We first restrict the parameter space to
$R_0\mathbb{S}^{d-1}\subseteq\Theta_{R_0}$, yielding
\begin{align}
    \sup_{\btheta^*\in\Theta_{R_0}}
    \mathbb{E}_{\btheta^*}
    \|\tilde{\btheta}-\btheta^*\|_2
    &\ge \sup_{\bv^*\in \mathbb{S}^{d-1}} \mathbb{E}_{R\bv^*}
    \|\tilde{\btheta}-R_0\bv^*\|_2
    \nn  
    \\
    &\stackrel{(a)}{\ge} 
    \frac{R_0}{2}
    \sup_{\bv^*\in\mathbb{S}^{d-1}}
    \mathbb{E}_{R_0\bv^*}
    \|\tilde{\bv}-\bv^*\|_2 \stackrel{(b)}{\ge}
    c\sqrt{R_0d/n}.
    \label{full-direction-lower}
\end{align}
where in $(a)$ we use
\ref{lem:lowerl2direction}, and $(b)$ is due to the direction estimation lower bound
\cite[Theorem 1]{hsu2024sample} that holds under $n\gtrsim d/R_0$. 


We next restrict the parameter space to
$\{R\be_1:R\in[R_0,2R_0]\}\subseteq\Theta_{R_0}$, where $\be_1$ denotes the first canonical basis vector. Similarly, by the reverse
triangle inequality and Theorem \ref{thm:minmaxlower},
\begin{align}
    \sup_{\btheta^*\in\Theta_{R_0}}
    \mathbb{E}_{\btheta^*}
    \|\tilde{\btheta}-\btheta^*\|_2
    \ge  \sup_{R\in[R_0,2R_0]}\mathbb{E}_{R\be_1}
    \|\tilde{\btheta}-R\be_1\|_2 \ge
    \sup_{R\in[R_0,2R_0]}
    \mathbb{E}_{R\be_1}
    \big|\|\tilde{\btheta}\|_2-R\big|\ge
    c\sqrt{R_0^3/n}.
    \label{full-norm-lower}
\end{align}
Combining \eqref{full-direction-lower}--\eqref{full-norm-lower} and
using $\max\{a,b\}\ge(a+b)/2$ proves \eqref{jointlower}. 
 \end{proof}
\begin{rem}[Interpretations of $\sqrt{R_0d/n}$ and $\sqrt{R_0^3/n}$]
\label{rem:full-minimax-interpretation}
The two terms in Theorem \ref{thm:fullminimax} arise from distinct
statistical difficulties. The lower bound
$\sqrt{R_0d/n}$ holds even when the parameter norm is fixed and known,
and therefore captures the difficulty of direction estimation. In
contrast, as already mentioned, the lower bound $\sqrt{R_0^3/n}$   captures the intrinsic difficulty of norm estimation.
\end{rem}

 \begin{theorem}[Improved error rate for the MLE]\label{thm:fullmle}
 Under Assumption \ref{assump1}, there exist universal constants $C>c>0$ such that, if $n\ge CRd$, then with probability at least $1-Ce^{-cd}-Cn^{-1}$, the MLE in (\ref{mlelogi}) exists and satisfies
 \[\|\hat{\btheta}_{\mle}-\btheta^*\|_2 \le C\bigg(\sqrt{\frac{Rd}{n}}+ \sqrt{\frac{R^3\log(nR)}{n}}+\frac{R^2d\log(nR)}{n}\bigg).\]
 \end{theorem} 
 \begin{proof}
     By the result of \cite{chardon2024finite} (as restated in Lemma \ref{lem:chardonbound}), the MLE exists and 
     \(\|\hat{\bv}_{\mle}-\bv^*\|_2\lesssim \sqrt{d/nR} \)
     holds with probability at least $1-e^{-d}$. By Theorem \ref{thm:upper}, we also have $|\hat{R}_{\mle}-R|\lesssim \sqrt{R^3\log(nR)/n}+R^2d\log(nR)/n$ with the promised probability. Now applying Lemma \ref{lem:dirplusnorm}   completes the proof.  
 \end{proof}
\begin{rem}[Optimality of the MLE] Our improved error rate for the MLE  in Theorem \ref{thm:fullmle} reduces to $\tilde{O}(\sqrt{Rd/n}+\sqrt{R^3/n})$ whenever $n\ge Rd\min\{d,R^2\}$, matching the minimax lower bound in Theorem \ref{thm:fullminimax}, up to logarithmic factors. Therefore, we conclude that the MLE is near minimax optimal under $n\ge Rd\min\{d,R^2\}.$ 
\end{rem}

We now combine $\hat{R}_{\rm DB}$ with the optimal direction estimator $\hat{\bv}_{\mle}$ to develop a minimax optimal estimator for full parameter estimation.  

 \begin{algorithm}[ht!]
\caption{A Minimax Optimal Estimator}
\label{alg:minimaxestimator}
\begin{algorithmic}[1]

\Require  
\(\{(\bx_i,y_i)\}_{i=1}^{n}
\subset \mathbb{R}^{d}\times\{0,1\}\) and the positive integers $(m_1,m_2,m_3)$ such that  $m_1+m_2+m_3=n$. Let $\calI_1:=[m_1]$, $\calI_2=[m_1+m_2]\setminus[m_1]$, $\calI_3= [n]\setminus [m_1+m_2]$

\State \textbf{Direction estimation:} Compute (\ref{mlelogi})   and set $\hat{\bv} = \hat{\btheta}_{\mle}/\|\hat{\btheta}_{\mle}\|_2$

\State \textbf{Norm estimation:} Run Algorithm \ref{alg:debiased-norm-estimate} to obtain $\hat{R}_{\rm DB}$

\Ensure \(\hat{\btheta}_{\rm DB}:=\hat{R}_{\mathrm{DB}}\hat{\bv}\).

\end{algorithmic}
\end{algorithm}

 \begin{theorem}[Minimax optimal estimator]\label{thm:fullminimaxupper} 
      Under Assumption \ref{assump1}, there exist universal constants $C>c>0$ such that the following holds. Given any $t\ge 1$, if $\min\{m_1,m_2,m_3\}\ge cn$, $n\ge CR(d+t)$, then with probability at least $1-C\exp(-ct)$, the MLEs in (\ref{mlelogi}) and (\ref{smallmle}) exist and the output of Algorithm \ref{alg:minimaxestimator} satisfies
      \[\|\hat{\btheta}_{\rm DB} -\btheta^*\|_2 \le C\bigg(\sqrt{\frac{R(d+t)}{n}}+\sqrt{\frac{R^3t}{n}}\bigg).\]      
 \end{theorem}
 \begin{proof}
     We omit the proof since it is parallel to that of Theorem \ref{thm:fullmle} except that we invoke Theorem \ref{thm:minimaxupper} (instead of Theorem \ref{thm:upper}) for the norm estimation.
 \end{proof}   
 \begin{rem}[Minimax optimal rate for parameter estimation]\label{rem:Efullerror}
Integrating the tail bound in Lemma \ref{lem:chardonbound}, with an arbitrary unit vector assigned whenever the MLE does not exist, gives
\(
    \mathbb{E}_{R\bv^*}
    \|\hat{\bv}_{\mle}-\bv^*\|_2
    \lesssim
    \sqrt{\frac{d}{nR}}.
\)
Combining this with Remark \ref{rem:Enormerror}, the   estimator
\(
    \hat{\btheta}_{\rm DB}^{\rm tr}
    :=
    \hat R_{\rm DB}^{\rm tr}\hat{\bv}_{\mle}\)
satisfies
\[
    \sup_{R_0\le\|\btheta^*\|_2\le 2R_0}
    \mathbb{E}_{\btheta^*}
    \|\hat{\btheta}_{\rm DB}^{\rm tr}-\btheta^*\|_2
    \lesssim
    \sqrt{R_0d/n}
    +
    \sqrt{R_0^3/n}
\]
under $n\gtrsim R_0d$. Together with Theorem \ref{thm:fullminimax}, this shows that the minimax optimal rate for parameter estimation over $\Theta_{R_0}:=\{\bu\in\mathbb{R}^d:R_0\le \|\bu\|_2\le 2R_0\}$ is at the order $\Theta(\sqrt{R_0d/n}+\sqrt{R_0^3/n})$.
 \end{rem}

\section{Numerical Simulations}\label{sec:numerics}
In this section we provide some numerical examples to demonstrate the bias of the MLE norm estimate  (Proposition \ref{prop:asymptotic-radial-bias}) and to show that the minimax optimal algorithms (Algorithms \ref{alg:debiased-norm-estimate}, \ref{alg:minimaxestimator}) outperform MLE in certain regimes. In these experiments, we set $R=\lfloor c_1n^{\rho_1}\rfloor$ and $d=\lfloor c_2n^{\rho_2}\rfloor$ and vary $n$ over a certain range. Each data point is averaged over $100 $ independent trials. All experiments were implemented using Matlab R2022a on
a laptop with an Intel CPU up to 2.5 GHz and 32 GB RAM. The Matlab codes are available in \href{https://github.com/junrenchen58/logistic-regression-mle-and-minimax}{\texttt{https://github.com/junrenchen58/logistic-regression-mle-and-minimax}}.

\subsection{Bias of the MLE Norm Estimate} 
Under fixed $R$ and as $\frac{d}{n}\to \delta$,    Proposition \ref{prop:asymptotic-radial-bias} predicts that 
$
    \hat{R}-R\sim \frac{d}{2Rq'(R)n}$ if $\delta$ is close enough to $0$. 
Also, for arbitrary $(n,d,R)$ satisfying $n\gtrsim Rd$, our Theorems \ref{thm:upper}, \ref{thm:minmaxlower} show that $|\hat{R}_{\mle}-R|=\tilde{O}(\sqrt{R^3/n}+ R^2d/n)$, and that the term $\sqrt{R^3/n}$ cannot be avoided by any algorithms. The main aim of our first experiment is to demonstrate that $
    \hat{R}_{\mle}-R\sim \frac{d}{2Rq'(R)n}$---albeit being established under high-dimensional asymptotics---remains accurate in finite-sample regimes. And as such, the term $R^2d/n$ in our upper bound for MLE appears to be the intrinsic bias rather than a proof artifact.

    Specifically, varying $n$ over $\{4000,8000,16000,32000,64000\}$, we test a fixed-$R$ setting with $R=5,~d= \lfloor 0.5n^{0.7}\rfloor$, where $R^2d/n\asymp n^{-0.3}$ dominates $\sqrt{R^3/n}\asymp n^{-0.5}$, and a growing-$R$ setting with $R= n^{0.15},~d=\lfloor 0.5n^{0.7}\rfloor$, where $R^2d/n\asymp n^{-0.1}$ dominates $\sqrt{R^3/n}\asymp n^{-0.275}$. In the figures, we will use $\hat{R}:=\hat{R}_{\mle}$. The results shown in Figure \ref{fig:normbias} suggest that the (signed) normalized quantity, $\frac{\hat{R}_{\mle}-R}{d/[2Rq'(R)n]}$, well approximates $1$, while another theoretical reference, $\frac{\sqrt{R^3/n}}{d/[2Rq'(R)n]}$, decays to $0$ in both regimes. The results suggest that the norm estimate given by the MLE incurs some dimension-dependent bias.  

\begin{figure}[ht!]
\centering
\includegraphics[width=0.92\textwidth]{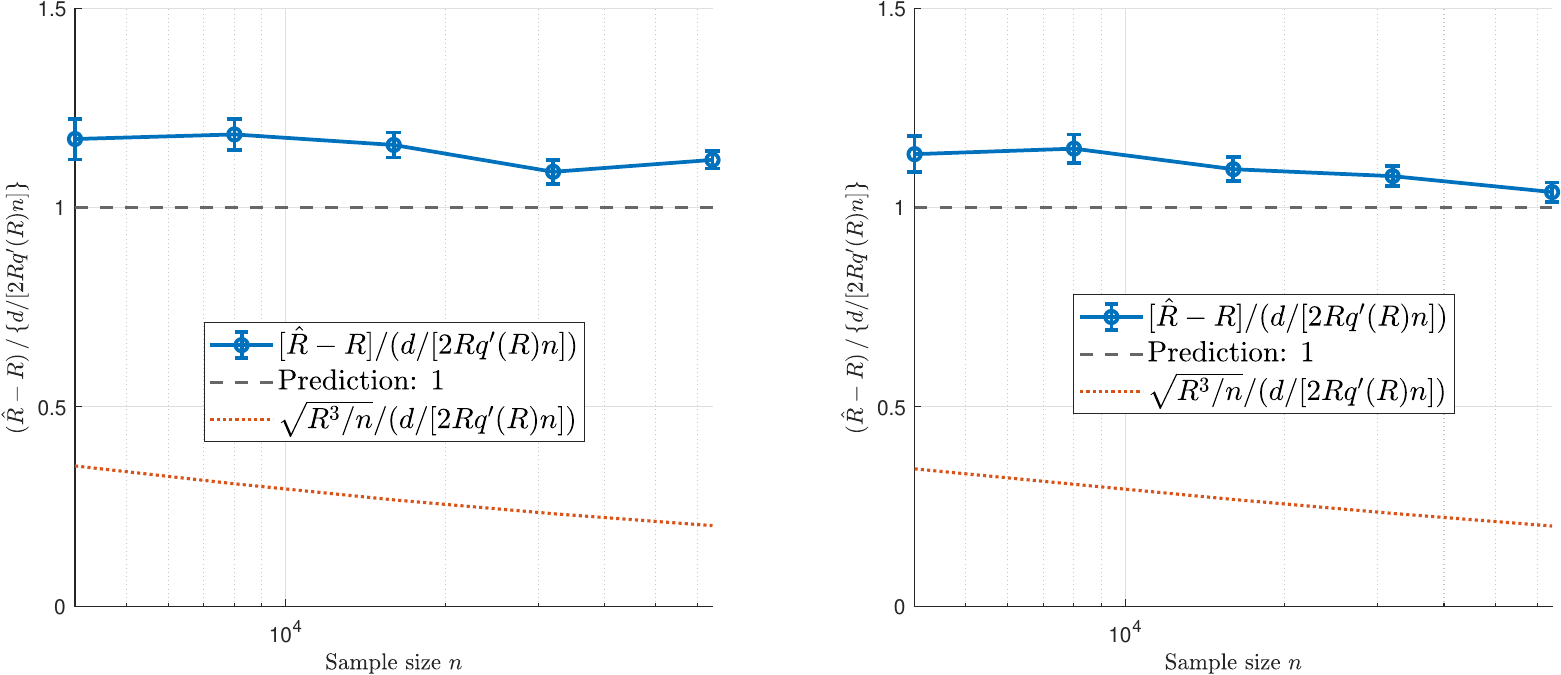}
\caption{The signed normalized quantity $\frac{\hat{R}_{\mle}-R}{d/[2Rq'(R)n]}$ in a fixed-$R$ setting (Left) and a growing-$R$ setting (Right). 
\label{fig:normbias}
}
\end{figure}

\subsection{Algorithms \ref{alg:debiased-norm-estimate}--\ref{alg:minimaxestimator} vs. MLE}
 Theoretically, Theorem \ref{thm:minimaxupper} shows that $\hat{R}_{\rm DB}$ computed from Algorithm \ref{alg:debiased-norm-estimate} is a minimax optimal norm estimator for logistic regression with Gaussian design. In contrast, our Theorem \ref{thm:upper} provides the best known error rate $\tilde{O}(\sqrt{R^3/n}+R^2d/n)$ for the MLE norm estimation. The main aim of our second experiment is to corroborate this by showing that, in some regime, the debiased estimator $\hat{R}_{\rm DB}$ noticeably outperforms the norm estimate given by MLE.

 Specifically, we set $R= \lfloor n^{0.005}\rfloor,~d=\lfloor 0.3n^{0.8}\rfloor$ and vary $n$
 over $\{6000,12000,24000,48000,96000\}$. In this setting, $\hat{R}_{\rm DB}$ achieves error rate $\sqrt{R^3/n}\asymp n^{-0.4925}$, while the performance bound for the norm of the MLE scales as $\sqrt{R^3/n}+ R^2d/n\asymp n^{-0.19}$. The results in Figure \ref{fig:mlevsminimax}(Left) are consistent with our theoretical findings---the proposed debiased norm estimator $\hat{R}_{\rm DB}$ achieves much lower errors than (the norm of) the MLE and exhibits sharper error decay rate with $n$. We also track the errors of Algorithm \ref{alg:minimaxestimator} and find that it also numerically improves on the MLE, as shown in Figure \ref{fig:mlevsminimax}(Right).

\begin{figure}[ht!]
\centering
\includegraphics[width=0.92\textwidth]{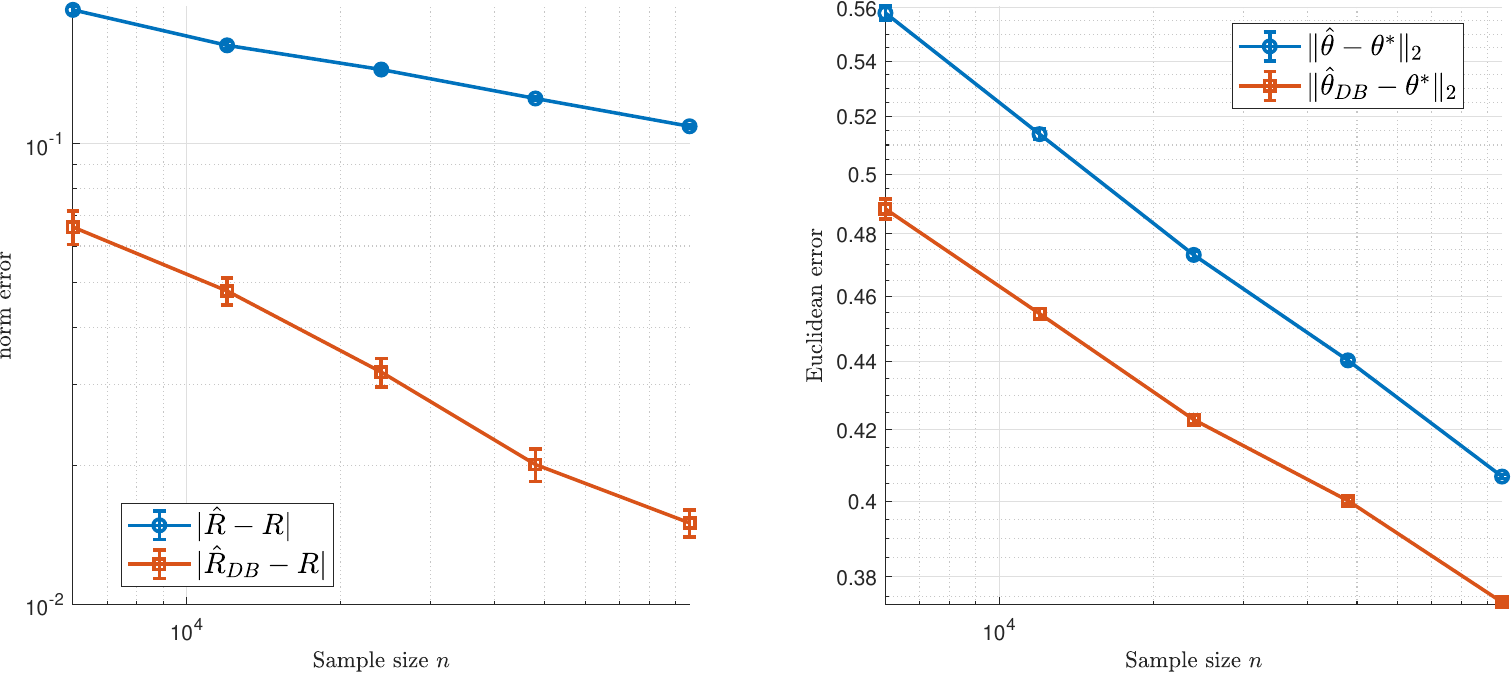}
\caption{Comparisons between Algorithms \ref{alg:debiased-norm-estimate}--\ref{alg:minimaxestimator} and MLE: $|\hat{R}_{\rm DB}-R|$ vs. $|\hat{R}-R|$ (Left) and $\|\hat{\btheta}_{\rm DB}-\btheta^*\|_2$ vs. $\|\hat{\btheta}_{\mle}-\btheta^*\|_2$ (Right). 
\label{fig:mlevsminimax}
}
\end{figure}

\section{Proof of Theorem \ref{thm:upper}}\label{sec:proofupper}

Recall that the empirical logistic loss $L_n(\btheta)$ is given in  (\ref{mlelogi}). We also let 
 $L(\btheta)=\mathbb{E} L_n(\btheta)$
be the population loss, and define the empirical gradient deviation as 
\begin{align}
    \bZ_n(\btheta):=\nabla L_n(\btheta)-\nabla L(\btheta)= \frac{1}{n}\sum_{i=1}^n (s(\bx_i^\top\btheta)-y_i)\bx_i - \mathbb{E}\big[(s(\bx_i^\top\btheta)-y_i)\bx_i\big].\label{zndef} 
\end{align}  
 Our upper bound builds on the following result from \cite{chardon2024finite}. 

\begin{lem}\cite[Theorem 1 \& Equation (19)]{chardon2024finite} \label{lem:chardonbound}Under Assumption \ref{assump1}, there exist universal constants $C>0$ such that, for any $t\ge 1$,  if $n\ge CR(d+t)$, then with probability at least $1-e^{-t}$, the MLE $\hat{\btheta}_{\mle}=\hat{R}_{\mle}\hat{\bv}_{\mle}$ defined in (\ref{mlelogi}) exists and satisfies
\(\|\hat{\bv}_{\mle}-\bv^*\|_2\le C\sqrt{\frac{d+t}{nR}}\) and \(|\hat{R}_{\mle}-R|\le C\sqrt{\frac{R^3(d+t)}{n}}\,.\)   
\end{lem}

From now on, we simply write $\hat{\bv}:=\hat{\bv}_{\mle}$ and $\hat{R}:=\hat{R}_{\mle}$. Under $n\gtrsim Rd$, Lemma \ref{lem:chardonbound} with $t=d$ yields that $\hat{\btheta}_{\mle}$ exists and satisfies
\begin{align}\label{chardonboundassume}
    \|\hat{\bv}-\bv^*\|_2 \le C\sqrt{\frac{d}{nR}} \qquad \textrm{and}\qquad |\hat{R}-R|\le C\sqrt{\frac{R^3d}{n}}
\end{align}
(for some universal constant $C$) with probability at least $1-e^{-d}$. We shall proceed on these high-probability events.  Note that $|\hat{R}-R|=O(\sqrt{R^3d/n})$ remains useful as it asserts that,  under $n\gtrsim Rd$, $\|\hat{\btheta}_{\mle}\|_2$ and $R$ are at the same order:   
\begin{align}\label{correctorder} 
   R/2\le \hat{R} \le 2R. 
\end{align}

\subsection{Deterministic Reduction}
Recall that $\bZ_n(\btheta)$ is the   gradient deviation defined in (\ref{zndef}). For \(r>0\) and \(\bv\in\mathbb S^{d-1}\), define
\begin{align}\label{defineSn}
     S_n(r,\bv)
    :=
    \langle \bZ_n(r\bv),\bv\rangle. 
\end{align}
On the localization event (\ref{correctorder}), the following lemma gives a deterministic bound on $|\hat{R}-R|$.

\begin{lem}
\label{lem:radial-score} 
In our setting, on the high-probability event of \(\frac{R}{2}\le \hat R\le 2R\) as in (\ref{correctorder}), it holds  that 
\[
    |\hat R-R|
    \le
    C R^3
    \big\{
        \|\hat \bv-\bv^*\|_2^2
        +
        |S_n(\hat R,\hat \bv)|
    \big\}.
\]
\end{lem}

\begin{proof}
Since \(\hat\btheta_{\mle}=\hat R\hat \bv\) is a stationary point of \(L_n\), we have
$
    \nabla L_n(\hat R\hat \bv)=\bm{0}.$ 
By definition of \(\bZ_n(\btheta)\) in (\ref{zndef}), we can also write
$
    \nabla L_n(\hat R\hat \bv)
    =
    \nabla L(\hat R\hat \bv)
    +
    \bZ_n(\hat R\hat \bv)$. By substituting $
    \nabla L_n(\hat R\hat \bv)=\bm{0}$ and $  \nabla L(\hat R\hat \bv)=  q(\hat R)\hat \bv
    -
    q(R)\bv^*$ from Lemma~\ref{lem:population-gradient}, we obtain 
\[
    q(\hat R)\hat \bv
    -
    q(R)\bv^*
    +
    \bZ_n(\hat R\hat \bv) = \bm{0}.  
\]
Taking inner product with \(\hat \bv\) and using the convention in (\ref{defineSn}) establish 
$
    q(\hat R)
    -
    q(R)\langle \bv^*,\hat \bv\rangle
    +
    S_n(\hat R,\hat \bv) =0.$
Rearranging   yields 
\begin{align}
     q(\hat R)-q(R)
    =
    -(1-\langle \bv^*,\hat{\bv}\rangle)q(R)
    -
    S_n(\hat R,\hat \bv), \label{firstorderequ}
\end{align}
and then triangle inequality further gives
\begin{align}\label{mainine}
    | q(\hat R)-q(R)| \le |q(R)|\big|1-\langle \bv^*,\hat \bv\rangle\big| + |S_n(\hat{R},\hat{\bv})|
\end{align}
By the mean value theorem and $\hat{R}\in [\frac{R}{2},2R]$, there exists \(\xi\in [\frac{R}{2},2R]\) such that
\(
    |q(\hat R)-q(R)|=|q'(\xi)||\hat R-R|.\)
Under $R\ge 1$, 
Lemma~\ref{lem:qprime} gives 
$q'(\xi)\asymp R^{-3},$ hence for some universal constant $c>0$, 
\begin{align}\label{taylor}
    |q(\hat{R})-q(R)| \ge \frac{c|\hat{R}-R|}{R^3}.
\end{align}
We now combine (\ref{mainine}) and (\ref{taylor}), then use $|q(R)|\le\sqrt{2/\pi}$ and  $|1-\langle \bv^*,\hat{\bv}\rangle|= \frac{\|\hat{\bv}-\bv^*\|_2^2}{2}$ (since $\hat{\bv},\bv^*\in \mathbb{S}^{d-1}$). 
The claim follows.
\end{proof}

In light of Lemma \ref{lem:radial-score}, we shall separately bound $R^3\|\hat{\bv}-\bv^*\|_2^2$ and $R^3|S_n(\hat{R},\hat{\bv})|$. By (\ref{chardonboundassume}), we already have \(
    R^3\|\hat{\bv}-\bv^*\|_2^2 \lesssim \frac{R^2d}{n},\)
which is within the desired  error bound (\ref{thm1norm}). The only remaining work is to control \(S_n(\hat R,\hat \bv) =\big\langle \nabla L_n(\hat{R}\hat{\bv})-\nabla L(\hat{R}\hat{\bv}),\hat{\bv}\big\rangle.\)

To this end, we start with the
decomposition
\begin{align}\label{decomposeSn}
     &S_n(\hat R,\hat \bv)
    =
    S_n(\hat R,\bv^*)
    +
    \big\{
        S_n(\hat R,\hat \bv)-S_n(\hat R,\bv^*)
    \big\}\\\label{decomposeSn2}
    \Longrightarrow\,& |S_n(\hat R,\hat \bv)|
    \le 
    |S_n(\hat R,\bv^*)|
    +
    |\big\{
        S_n(\hat R,\hat \bv)-S_n(\hat R,\bv^*)
    \big\}| 
\end{align}
For fixed $r>0$ and $\bv\in \mathbb{S}^{d-1}$, by (\ref{zndef}) and (\ref{defineSn}), \(S_n(r,\bv)\)    
is the mean of i.i.d. centered random variables. Yet $S_n(\hat{R},\bv^*)$ and $S_n(\hat{R},\hat{\bv})-S_n(\hat{R},\bv^*)$ in (\ref{decomposeSn}) can no longer be viewed as the mean of independent variables since $\hat{R},\hat{\bv}$ depend on the data. Our remedy is to take  uniform bounds over the local set $\| \bv-\bv^*\|_2\le \delta$ for some $\delta\asymp \sqrt{d/(nR)}$, $\hat{R}\in [R/2,2R]$. This is justified by (\ref{chardonboundassume})--(\ref{correctorder}). We thus arrive at 
\begin{gather}\label{normsup1}
    |S_n(\hat{R},\bv^*)|\le \sup_{  r\in[R/2,2R]}|S_n(r,\bv^*)|,\\
    |S_n(\hat{R},\hat{\bv})-S_n(\hat{R},\bv^*)|\le \sup_{  r\in[R/2,2R]}\sup_{\bv\in\calV_\delta}|S_n(r,\bv)-S_n(r,\bv^*)|, \label{dirincre}
\end{gather}
where we define $\mathcal V_\delta
    :=
    \{
        \bv\in\mathbb S^{d-1}:
        \|\bv-\bv^*\|_2\le \delta\}$.

\subsection{Bounding the Localized Empirical Processes}\label{sec:boundprocessthm2}
In this subsection, we control the empirical processes arising in (\ref{normsup1})--(\ref{dirincre}). We first use  Bernstein's inequality to establish nonuniform bounds (for fixed $r,\bv$), and then strengthen these to uniform bounds via a covering argument. Here, to invoke Bernstein's inequality, we have to validate the moment bounds in Lemma \ref{lem:bernstein210}. This involves some elementary but technical Gaussian integral estimates; they are collected in Appendix \ref{app:integral}. Once the non-uniform bounds are in place, the covering arguments that extend it to a uniform bound are fairly standard and present no extra subtlety. As such, we relegate the covering arguments to Appendix \ref{app:covering}.

\begin{lem}[Bounding $\sup_r|S_n(r,\bv^*)|$]
\label{lem:true-direction-process}
For some absolute constant $C$, if $n\ge CR$, then with probability at least
\(1-\frac{C}{n}\), 
\[
    \sup_{r\in[R/2,2R]}
    |S_n(r,\bv^*)|
    \le
    C\bigg\{
        \sqrt{\frac{\log(nR)}{nR^3}}
        +
        \frac{\log(nR)}{nR}
    \bigg\}.\]
\end{lem}

\begin{proof}
Note that  \(
    S_n(r,\bv^*) = \frac{1}{n}\sum_{i=1}^n\big(\big(s(r\bx_i^\top\bv^*)-y_i\big)\bx_i^\top\bv^* - \mathbb{E}\big[\big(s(r\bx_i^\top\bv^*)-y_i\big)\bx_i^\top\bv^*\big]\big).
\)
Let $
    g_i=\bx_i^\top\bv^*,$ then we have \(g_i\sim N(0,1)\) and
$
    y_i\mid g_i\sim \mathrm{Bernoulli}(s(Rg_i)),$ and hence 
    \begin{align} \label{Snexpandver}
        S_n(r,\bv^*) = \frac{1}{n}\sum_{i=1}^n\Big(\big(s(rg_i)-y_i\big)g_i-\mathbb{E}\big[(s(rg_i)-s(Rg_i))g_i\big]\Big) 
    \end{align}    
For fixed $r\in[\frac{R}{2},2R]$, (\ref{eq:centered-radial-moment}) in Lemma \ref{lem:1dprobound} gives 
\[
    \mathbb{E}\big|(s(rg_i)-y_i)g_i - \mathbb{E}[(s(rg_i)-y_i)g_i]\big|^p \le \frac{C^pp!}{R^{p+1}},\quad \textrm{for any integer } p\ge 2.
\] 
Therefore, in view of (\ref{Snexpandver}), Bernstein's inequality (Lemma \ref{lem:bernstein210}) yields that, for any fixed $r\in [\frac{R}{2},2R]$ and $t>0$,
\begin{align}\label{nonuniform1}
    \mathbb{P}\bigg(|S_n(r,\bv^*)|\ge C\bigg[\sqrt{\frac{t}{nR^3}}+\frac{t}{nR}\bigg] \bigg) \le 2e^{-t}.
\end{align}
This immediately implies that the non-uniform bound $S_n(r,\bv^*)\lesssim \sqrt{\frac{\log n}{nR^3}}+\frac{\log n}{nR}$ holds with probability at least $1-2/n$. It turns out that, by a standard covering argument over $[R/2,2R]$, we establish a comparable bound as claimed in the lemma. See Appendix \ref{app:lem53covering}.
\end{proof}

\begin{lem}[Bounding $\sup_{r,\bv}|S_n(r,\bv)-S_n(r,\bv^*)|$]
\label{lem:localized-increment} 
There exist some universal constants $ C>c>0$ such that the following holds.    If $n\ge d$ and  
$R\delta\le c$,
then with probability at least \(1-Ce^{-cd}\),    
\[
    \sup_{r\in[R/2,2R]}
    \sup_{\bv\in\mathcal V_\delta}
    |S_n(r,\bv)-S_n(r,\bv^*)|
    \le C\bigg\{
        \delta\sqrt{\frac{d\log(nR)}{nR}}
        +
        \delta\frac{d\log(nR)}{n}
    \bigg\}.\] 
\end{lem}

\begin{proof}
For \(r>0\) and \(\bv\in\mathbb S^{d-1}\), define 
$
     \psi_i(r,\bv)
    :=
    \{s(r \bx_i^\top\bv)-y_i\}\bx_i^\top\bv$ and $
    m(r,\bv):=\mathbb{E}[\psi_i(r,\bv)].$ Then
$S_n(r,\bv)=
    \frac1n\sum_{i=1}^n
    \{\psi_i(r,\bv)-m(r,\bv)\},$  and in turn we can write 
\begin{align}\label{differiidsum}
      S_n(r,\bv)-S_n(r,\bv^*)
    =
    \frac1n\sum_{i=1}^n
    \left[
        \Delta_i(r,\bv)-\mathbb{E}\Delta_i(r,\bv)
    \right], \quad\textrm{where}~~\Delta_i(r,\bv):=\psi_i(r,\bv)-\psi_i(r,\bv^*). 
\end{align}
As such, we seek to bound $\sup_{r\in[R/2,2R]}
    \sup_{\bv\in\mathcal V_\delta}\big| \frac1n\sum_{i=1}^n
    \left[
        \Delta_i(r,\bv)-\mathbb{E}\Delta_i(r,\bv)
    \right]\big|.$ Fix \(r\in[R/2,2R]\) and \(\bv\in\mathcal V_\delta\), we let $\eta=\|\bv-\bv^*\|_2.$ Then Lemma \ref{lem:directincrebound} yields
 \begin{align}
     \mathbb{E}|\Delta_i(r,\bv)|^p
    \le
    C^p p!\frac{\eta^p}{R},~~\textrm{for any integer \(p\ge 2\)}. 
   \label{desiredmoment}
\end{align}
In view of this and using $\eta \le\delta$, for any fixed pair of $(r,\bv)\in [R/2,2R]\times \calV_\delta$, Lemma \ref{lem:bernstein210} yields 
\begin{align}\label{fixedrvbound}
   \mathbb{P}\bigg(\big|S_n(r,\bv)-S_n(r,\bv^*)\big|
    \le C\bigg\{
        \delta\sqrt{\frac{t}{nR}}+\delta\frac{t}{n}
    \bigg\}\bigg)\ge 1-2e^{-t},\quad \forall t>0.
\end{align}
By setting $t\asymp\log n$, this implies that the   non-uniform bound $|S_n(r,\bv)-S_n(r,\bv^*)| \lesssim \delta\sqrt{\frac{\log n}{nR}}+\delta\frac{\log n}{n}$ holds with at least $1-2/n$ probability. However, what we desire is the uniform bound over $(r,\bv)\in [R/2,2R]\times \calV_\delta$, and therefore a covering argument should be invoked. As it  turns out, this worsens the non-uniform bound with extra dependence on $d$ and yields the claimed bound in the lemma; see Appendix \ref{app:lem54covering}.
\end{proof}

\subsection{Putting Pieces Together}
We now put all the pieces together to conclude the proof of Theorem \ref{thm:upper}. 
\begin{proof}[Proof of Theorem \ref{thm:upper}] 
By Lemma~\ref{lem:radial-score} and (\ref{decomposeSn2})--(\ref{dirincre}),  
\begin{align}\label{normuppergive}
      |\hat R-R|
    \lesssim R^3
    \bigg\{
        \|\hat \bv-\bv^*\|_2^2
        +
        \sup_{  r\in[R/2,2R]}|S_n(r,\bv^*)| + \sup_{  r\in[R/2,2R]}\sup_{\bv\in\calV_\delta}|S_n(r,\bv)-S_n(r,\bv^*)|
    \bigg\}
\end{align}
holds for some $\delta\asymp \sqrt{d/(nR)}$. 
Under $n\gtrsim Rd$,  $R\delta\asymp \sqrt{Rd/n}$ is sufficiently small, hence  
 Lemma~\ref{lem:localized-increment} applies. Substituting $\|\hat{\bv}-\bv^*\|_2^2\lesssim\frac{d}{nR}$   from (\ref{chardonboundassume}) and the bounds in Lemmas \ref{lem:true-direction-process}--\ref{lem:localized-increment} into (\ref{normuppergive}) yields  (\ref{thm1norm}). The proof is complete. 
\end{proof}

\section{Proof of Theorem \ref{thm:minimaxupper}}\label{sec:proofminimaxupper}
Following the developments in Section \ref{sec:minimaxnorm}, this section proves Theorem \ref{thm:minimaxupper}. 
\subsection{Locating the Estimate in (\ref{1dmle})}\label{sec:local1dmle}
We have already shown in Lemma \ref{lem:population-bias-1dmle} that $\hat{r}$ found by (\ref{1dmle}) approximates $r^\circ$ rather than $R$. In the following, we further control $|\hat{r}-r^\circ|$ by a convex localization argument. 
\begin{lem} 
\label{lem:profile-local}
Conditioning on $(\breve{\bv},\breve{R})$ and using the randomness of $\{(\bx_i,y_i)\}_{i\in \calI_2}$, there exist absolute constants $C>c>0$ such that the following holds. Given any $t\ge 1$, if $\breve{\bv}^\top\bv^*\ge 0$, $\sqrt{1-(\breve{\bv}^\top\bv^*)^2}\le cR^{-1}$, $\frac{R}{2}\le \breve{R}\le 2R$ and $m_2\ge CRt$, then with
probability at least $1-4e^{-t}$, 
\begin{equation}
  |\hat r-r^\circ|
  \le
  C 
    \sqrt{\frac{R^3t}{m_2}}.
  \label{eq:profile-local-bound}
\end{equation} 
\end{lem}

\begin{proof}
Let $\ell(r)=\frac{1}{m_2}\sum_{i\in\calI_2}\big[\log(1+\exp(r\bx_i^\top\breve{\bv}))-y_ir\bx_i^\top\breve{\bv}\big]$ be the loss function in (\ref{1dmle}).
For a fixed $r\in[R/4,2R]$, we have 
\(
  \ell'(r)
  :=\frac{1}{m_2}\sum_{i\in\calI_2}
  \{s(r\bx_i^\top\breve{\bv})-y_i\}\bx_i^\top\breve{\bv}.\) It was shown in the proof of Lemma \ref{lem:population-bias-1dmle} that 
  \begin{align}
      \mathbb{E}\ell'(r)=q(r)-q(r^\circ).\label{Eellpi}
  \end{align} We let
\(
  X_i(r):=\{s(r\bx_i^\top\breve{\bv})-y_i\}\bx_i^\top\breve{\bv}-\mathbb{E}\ell'(r) \) so that $\ell'(r)-\mathbb{E}\ell'(r)=\frac{1}{m_2}\sum_{i\in\calI_2}X_i(r)$.
The variables $X_i(r)$ are centered and independent. To apply Lemma \ref{lem:1dprobound}, write
\(
    \alpha=\langle \breve{\bv},\bv^*\rangle,
    ~
    \beta=\sqrt{1-\alpha^2},\)
and, when $\beta>0$, let $\bu=(\bv^*-\alpha\breve{\bv})/\beta$; when $\beta=0$, choose any unit vector $\bu\perp\breve{\bv}$. Then
$Z_i=\bx_i^\top\breve{\bv}$ and $T_i=\bx_i^\top\bu$ are independent standard normal variables, while
\(
    y_i\mid (Z_i,T_i)
    \sim \Bern\big(s(R[\alpha Z_i+\beta T_i])\big).\)
Since $\beta\le cR^{-1}$ and $r\in[R/4,2R]$, the centered radial moment bound in Lemma \ref{lem:1dprobound} applies. Indeed, with $\varepsilon_r=Y_i-s(rZ_i)$,
$X_i(r)=-(Z_i\varepsilon_r-\mathbb{E}[Z_i\varepsilon_r])$. Therefore, for every integer $q\ge2$,
\(
  \mathbb{E}|X_i(r)|^q
  \le\frac{C^qq!}{R^{q+1}}.
\) 
Consequently,  Lemma~\ref{lem:bernstein210} yields
\begin{equation}
  \mathbb{P}\bigg(
    \big|\ell'(r)-\mathbb{E}\ell'(r)\big|
    >C\bigg\{
      \sqrt{\frac{t}{m_2R^3}}+\frac{t}{m_2R}
    \bigg\} 
  \bigg)
  \le2e^{-t}.
  \label{eq:profile-score-tail}
\end{equation}
Set $h:=C_0\big\{
    \sqrt{\frac{R^3t}{m_2}}+\frac{R^2t}{m_2}
  \big\}$ for a sufficiently large universal constant $C_0$, and because of $m_2\ge CRt$,
we have $h\le R/4$. Moreover, by Equation (\ref{eq:population-attenuation}) in Lemma \ref{lem:population-bias-1dmle}, we have $R-r^\circ\asymp R^3(1-\breve{\bv}^\top\bv^*)$. Combining with $\sqrt{1-(\breve{\bv}^\top\bv^*)^2}\le cR^{-1}$ for some small enough $c$,  we have $\frac{3R}{4}\le r^\circ\le R$. Therefore, $r^\circ\pm h\in[R/2,5R/4]$.  Since
$\breve R\ge R/2$, $r^\circ\pm h$ lie in the feasible interval
$[0,4\breve R]$.  By 
\eqref{Eellpi}, Lemma \ref{lem:qprime} and Taylor's theorem, there exist some $\xi_1,\xi_2\in[R/2,5R/4]$ such that 
\begin{align*}
   \mathbb{E}\ell'(r^\circ+h)
  &=q(r^\circ+h)-q(r^\circ)= q'(\xi_1)h\ge c_1hR^{-3},
  \\
    \mathbb{E}\ell'(r^\circ-h)
  &=q(r^\circ-h)-q(r^\circ)
  =-q'(\xi_2)h\le-c_1hR^{-3}
\end{align*}
for some universal constant $c_1>0.$ Apply \eqref{eq:profile-score-tail} to $r^\circ+h$ and $r^\circ-h$, then with probability at least $1-4\exp(-t)$, a large enough $C_0$ ensures 
\begin{gather*}
    \ell'(r^\circ+h) \ge \mathbb{E}  \ell'(r^\circ+h)-|\ell'(r^\circ+h)- \mathbb{E}  \ell'(r^\circ+h)|>0,\\
    \ell'(r^\circ-h) \le\mathbb{E}\ell'(r^\circ-h) + | \ell'(r^\circ-h) -\mathbb{E}\ell'(r^\circ-h)| <0. 
\end{gather*}
Note that  $\ell'(r)$ is
nondecreasing due to the convexity of $\ell(r)$.  Its constrained minimizer therefore lies in
$[r^\circ-h,r^\circ+h]$. Combining with $m_2\gtrsim Rt$, we arrive at  
\eqref{eq:profile-local-bound}. 
\end{proof}

\subsection{Concentration of the Debiasing U-Statistic}
Next, we establish the concentration bound for the $U$-statistic $\hat{B}$ to concentrate about $\mathbb{E}\hat{B}=(1-\alpha^2)q(R)^2=\beta^2q(R)^2$, where we write $\beta=\sqrt{1-\alpha^2}$. To summarize, we first decompose the concentration error into three terms $L_m+Q_m^{\mathrm{sc}}+Q_m^{\mathrm G}$; see (\ref{eq:Hoeffding})--(\ref{eq:Q-def}) and  (\ref{eq:L-exact})--(\ref{eq:Q-split}). The techniques for $L_m$ and $Q_m^{\mathrm{sc}}$ are routine applications of Bernstein's inequality (Lemma \ref{lem:bernstein210}), built upon certain moment bounds developed in Appendix \ref{app:integral}. Since these techniques have appeared in Sections \ref{sec:boundprocessthm2} and \ref{sec:local1dmle}, we relegate the mathematical details for bounding $L_m$ and $Q_m^{\mathrm{sc}}$ to Appendix \ref{app:bernsteinminimax} and only note the final bounds in the main paper. In contrast, controlling $Q_m^{\rm G}$ requires new techniques and will be the focus of the following proof.

\begin{lem}
\label{lem:direct-U}
Conditioning on $(\breve{\bv},\hat{r})$ and using the randomness of $\{(\bx_i,y_i)\}_{i\in\calI_3}$, there exist universal constants $C>c>0$ such that the following holds. Given any $t\ge 1$, if $\langle \breve{\bv},\bv^*\rangle\ge 0$, $ 1-(\breve{\bv}^\top\bv^*)^2\le \min\{\frac{C(d+t)}{m_3R},\frac{c}{R^2}\}$, $\hat{r}\in [\frac{R}{2},\frac{3R}{2}]$, $m_3\ge CRt$, 
then with probability at least $1-Ce^{-ct}$, 
\begin{equation}   
    |\hat B-q(R)^2\beta^2|
    <C\frac{\sqrt{dt}+t}{m_3R}. 
  \label{eq:direct-U-bound}
\end{equation}
\end{lem}
\begin{proof}
For notational simplicity, relabel the indices in $\calI_3$ as
$i=1,\ldots,m$ (we let $m=m_3$). In view of Lemma \ref{lem:EwEB}, let $\bm{\mu}:=\mathbb{E}\bw_i=q(R)\bP^\perp_{\breve{\bv}}\bv^*$. Expanding \eqref{hatEconstruct} with
$\bw_i=\bm{\mu}+(\bw_i-\bm{\mu})$ gives 
\begin{equation}
  \hat B-q(R)^2\beta^2=\hat B-\norm{\bm{\mu}}_2^2
  =L_m+Q_m,
  \label{eq:Hoeffding}
\end{equation}
where
\begin{align}
  L_m
  &:=\frac2m\sum_{i=1}^m
  \langle\bm{\mu},\bw_i-\bm{\mu}\rangle,
  \label{eq:L-def}
  \\
  Q_m
  &:=\frac{1}{m(m-1)}
  \sum_{i\ne j}
  \langle\bw_i-\bm{\mu},\bw_j-\bm{\mu}\rangle.
  \label{eq:Q-def}
\end{align}

We pause to introduce some notations. If $\bP^\perp_{\breve{\bv}}\bv^*\ne 0$, then we let $\bu=\frac{\bP^\perp_{\breve{\bv}}\bv^*}{\|\bP^\perp_{\breve{\bv}}\bv^*\|_2}$; otherwise, we choose $\bu$ as any vector orthogonal to $\breve{\bv}$. We then construct an orthogonal matrix $\bP\in \mathbb{R}^{d\times d}$ with the first two columns being $\breve{\bv}$ and $\bu$, denoted by 
\(\bP=[\breve{\bv},\bu,\bP_{\perp}]\)
for some $\bP_{\perp}\in \mathbb{R}^{d\times (d-2)}$. We then introduce 
\begin{align}
    \bg_i=\bP^\top\bx_i = [\bx_i^\top\breve{\bv},\,\bx_i^\top\bu,\,\bx_i^\top\bP_\perp]^\top:=[Z_i,T_i,\tilde{\bg}_i^\top]^\top.
\end{align}
By rotational invariance, $\bg_i\sim N(0,\bI_d)$, and hence $Z_i\sim N(0,1)$, $T_i\sim N(0,1)$ and $\tilde{\bg}_i\sim N(0,\bI_{d-2})$ are independent. By the construction of $\bu$, we have
\(\bv^*= \langle \bv^*,\breve{\bv}\rangle\breve{\bv}+ \bP_{\breve{\bv}}^\perp \bv^* =\alpha\breve{\bv} +\beta\bu.\)
Together with $\btheta^*=R\bv^*$ and the above notations,  this gives
\(
    \bx_i^\top \btheta^* = R\bx_i^\top\bv^* = R\big(\alpha \bx_i^\top\breve{\bv}+\beta \bx_i^\top\bu\big) = R(\alpha Z_i+\beta T_i).\)
Therefore, conditioning on $\bx_i$ and $\breve{\bv}$, we have
\(
    y_i\sim \Bern\big(s(R[\alpha Z_i+\beta T_i])\big).\)
Substituting these into (\ref{uiconstruct}), 
\[
    \bw_i =\big[y_i-s(\hat{r}Z_i)\big]\bP^\perp_{\breve{\bv}}\bx_i =\big[y_i-s(\hat{r}Z_i)\big]\bP^\perp_{\breve{\bv}}\bP\bg_i =  \big[y_i-s(\hat{r}Z_i)\big] \big(T_i\bu+\bP_\perp\tilde{\bg}_i\big).
\]
For convenience, we further introduce $e_i=y_i-s(\hat{r}Z_i)$  so that 
\(
    \bm{\mu}=q(R)\bP^\perp_{\breve{\bv}}\bv^* = q(R)\beta \bu,\)
then we can write $\bw_i= e_i(T_i\bu+\bP_\perp\tilde{\bg}_i)$ and 
\[
    \bw_i-\bm{\mu} = \underbrace{(e_iT_i-q(R)\beta)}_{:=\eta_i}\bu + e_i \bP_\perp\tilde{\bg}_i. 
\]
We now substitute these expressions into (\ref{eq:L-def})--(\ref{eq:Q-def}). By $\bu^\top\bP_\perp = 0$, 
\begin{align}
    L_m = \frac{2}{m}\sum_{i=1}^m\langle q(R)\beta\bu,\eta_i \bu + e_i\bP_\perp \tilde{\bg}_i\rangle = \frac{2q(R)\beta}{m} \sum_{i=1}^m \eta_i . \label{eq:L-exact}
\end{align}
Also, 
\begin{align}
     Q_m
  &=\frac{1}{m(m-1)}\sum_{i\ne j}\big\langle \eta_i\bu+e_i\bP_\perp \tilde{\bg}_i,\eta_j\bu+e_j\bP_\perp \tilde{\bg}_j\big\rangle\\ \label{eq:Q-split}
  &=\frac{1}{m(m-1)}\sum_{i\ne j}\eta_i\eta_j+ \frac{1}{m(m-1)}
  \sum_{i\ne j}e_ie_j\tilde{\bg}_i^\top\tilde{\bg}_j:=Q_m^{\mathrm{sc}}+Q_m^{\mathrm G},
\end{align}
We control the three terms $L_m,\,Q_m^{\mathrm{sc}},\,Q_m^{\mathrm G}$ separately.

\paragraph{Bounding $L_m$ and $Q_m^{\rm sc}$.} Built upon certain moment bounds in Appendix \ref{app:integral}, these are routine applications of   Bernstein's inequality. We relegate the details to Appendix \ref{app:bernsteinminimax} and simply announce the final bounds here:  with probability at least $1-Ce^{-ct}$, 
\begin{gather}
    |L_m|\lesssim\frac{\sqrt{dt}+t}{mR}, 
  \label{eq:L-final}\\
   |Q_m^{\mathrm{sc}}|
  \lesssim\frac{t}{mR}.
  \label{eq:Qsc-final}
\end{gather}

\paragraph{Bounding $Q_m^{\rm G}$.} It remains to bound $Q_m^{\rm G} = \frac{1}{m(m-1)}\sum_{i\ne j}e_ie_j\tilde{\bg}_i^\top\tilde{\bg}_j$. Conditioning on $\breve{\bv},\hat{r}$ and further on $\sigma\{Z_i,T_i,y_i\}_{i=1}^m$, 
 the  weights $ \be:=(e_1,\ldots,e_m)^\top$ are
fixed, whereas $\tilde{\bg}_1,\ldots,\tilde{\bg}_m$ remain i.i.d.
$N(0,\bI_{d-2})$ random vectors. Define the symmetric matrix
\(
  \bA_e:=\be\be^\top-\diag(e_1^2,\ldots,e_m^2) = \big(e_ie_j\mathbf{1}\{i\ne j\}\big)_{i,j}\in\mathbb{R}^{m\times m}. 
\)
For each coordinate $\ell=1,\ldots,d-2$, define
\(
  \bw^{(\ell)}
  :=(\tilde{g}_{1\ell},\ldots,\tilde{g}_{m\ell})^\top\in\mathbb{R}^m,\)
then
\[
  \sum_{i\ne j}e_ie_j\tilde{\bg}_i^\top\tilde{\bg}_j
  = \sum_{i\ne j}e_ie_j\bigg(\sum_{\ell=1}^{d-2}\tilde{g}_{i\ell}\tilde{g}_{j\ell}\bigg)=\sum_{\ell=1}^{d-2}
   (\bw^{(\ell)})^\top \bA_e\bw^{(\ell)}. 
\]
Stack the vectors $\bw^{(\ell)}$ into
$\bz\in\mathbb{R}^{m(d-2)}$.  By 
$\bz\sim N(0,\bI_{m(d-2)})$, and
\begin{equation}
  \sum_{i\ne j}e_ie_j\tilde{\bg}_i^\top\tilde{\bg}_j
  =\bz^\top \bB_e\bz,
  \quad
  \bB_e:=\bI_{d-2}\otimes \bA_e,
  \label{eq:QG-quadratic}
\end{equation}
where $\otimes$ denotes the Kronecker product; we use
$\|\cdot\|_{\rm op}$ and $\|\cdot\|_{\rm F}$ for the operator and
Frobenius norms, respectively. 
The diagonal of $\bA_e$ is zero, so
\(\tr(\bB_e)=\tr(\bI_{d-2})\tr(\bA_e)=0,\)
hence the quadratic form in \eqref{eq:QG-quadratic} is already
conditionally centered. Put
\(S_e:=\sum_{i=1}^m e_i^2.\) The residual-energy bound in Lemma \ref{lem:1dprobound} gives
$\mathbb{E}[e_i^2]\le C/R$. Moreover, $0\le e_i^2\le1$. For
$X_i:=e_i^2-\mathbb{E}[e_i^2]$, we have, for every integer $p\ge2$,
\[
  \mathbb{E}|X_i|^p
  \le\mathbb{E}X_i^2
  \le\mathbb{E}e_i^4
  \le\mathbb{E}e_i^2
  \le\frac CR.\]
Therefore, Lemma~\ref{lem:bernstein210} gives
\begin{equation}
  S_e = \sum_{i=1}^mX_i + m\mathbb{E}[e_i^2] 
  \lesssim\bigg\{\frac mR+\sqrt{\frac{mt}{R}}+t\bigg\}
  \lesssim \frac mR
  \label{eq:Se-bound}
\end{equation}
holds with probability at least  $1-2e^{-t}$, where
the last inequality uses $m\ge CRt$. On the event \eqref{eq:Se-bound}, we have
\(
  \norm{\bA_e}_{\mathrm F}^2
   =\sum_{i\ne j}e_i^2e_j^2
  \le S_e^2\lesssim \frac{m^2}{R^2}\), which implies \(\|\bA_{e}\|_{\rm op}\le\norm{\bA_e}_{\mathrm F} \lesssim\frac{m}{R}.\)
 In light of $\bB_e:=\bI_{d-2}\otimes \bA_e$, we have
\(
  \norm{\bB_e}_{\mathrm F}
 =\sqrt{d-2}\,\norm{\bA_e}_{\mathrm F}
  \lesssim\frac{m\sqrt d}{R}\)
  and \(
  \norm{\bB_e}_{\rm op}
  =\norm{\bA_e}_{\rm op}
  \lesssim\frac mR.\)  
By Lemma~\ref{lem:hanson-wright}, 
it holds with probability at least $1- 2e^{-t}$ that 
\(|\bz^\top \bB_e\bz|
  \lesssim \frac mR\{\sqrt{dt}+t\}\), 
which then gives 
\begin{equation}
  |Q_m^{\mathrm G}|= \bigg|\frac{\bz^\top \bB_e\bz}{m(m-1)}\bigg|
  \lesssim \frac{\sqrt{dt}+t}{mR}.
  \label{eq:QG-final}
\end{equation}
Finally, combining \eqref{eq:Hoeffding}, \eqref{eq:L-final},
\eqref{eq:Qsc-final}, and \eqref{eq:QG-final}, along with a union bound, completes the proof.
\end{proof}

\subsection{Completing the Proof}
We are now ready to establish Theorem \ref{thm:minimaxupper}.

\begin{proof}[Proof of Theorem \ref{thm:minimaxupper}]
Since $\min\{m_1,m_2,m_3\}\ge cn$, we have $m_1,m_2,m_3\asymp n$ and express the subsequent bounds in terms of $n$. In Step 1 of Algorithm \ref{alg:debiased-norm-estimate}, Lemma \ref{lem:chardonbound} gives
\begin{align}
    \|\breve{\bv}-\bv^*\|_2
    \lesssim
    \sqrt{\frac{d+t}{nR}},
    \qquad
    |\breve R-R|
    \lesssim
    \sqrt{\frac{R^3(d+t)}{n}}.
    \label{eq:stage-one-localization}
\end{align}
Let $\alpha:=\langle\breve{\bv},\bv^*\rangle$ and $\beta^2:=1-\alpha^2$. Under $n\ge CR(d+t)$, (\ref{eq:stage-one-localization}) implies
\(
    \breve R\in[R/2,2R],
    ~
    \alpha\ge0,
    ~
    \beta^2
    \lesssim
    \frac{d+t}{nR}
    \le
    \frac{c}{R^2}.
\)
Therefore, Lemmas \ref{lem:profile-local} and \ref{lem:direct-U} apply to Steps 2--3 and yield, on an event of probability at least $1-Ce^{-ct}$,
\begin{align}
    |\hat r-r^\circ|
    &\lesssim
    \sqrt{\frac{R^3t}{n}},
    \qquad
    \big|\hat B-\beta^2q(R)^2\big|
    \lesssim
    \frac{\sqrt{dt}+t}{nR}.
    \label{eq:two-stage-errors}
\end{align}
Since $r^\circ\in[3R/4,R]$, the first bound and $n\gtrsim Rt$ imply $\hat r\in[R/2,3R/2]$. Moreover, the boundedness of $q$ and Lemma \ref{lem:qprime} give
\(
    |q(\hat r)^2-q(r^\circ)^2|
    \lesssim
    \sqrt{\frac{t}{nR^3}}.
\) 
Using $q(r^\circ)=\alpha q(R)$ and $\alpha^2+\beta^2=1$, we obtain
\[
    |\hat S-q(R)^2|
      \le
    |q(\hat r)^2-q(r^\circ)^2|
    +
    |\hat B-\beta^2q(R)^2|  \lesssim
    \sqrt{\frac{t}{nR^3}}
    +
    \frac{\sqrt{dt}+t}{nR}.\]

Since $R\ge1$, we have $q(R)\ge q(1)>0$. Also, $\hat r\in[R/2,3R/2]$ implies $q(R)\in[q(\hat r/2),q(2\hat r)]$. Hence, by the projection in (\ref{hatqdefined}) and the inequality
\(
|\sqrt{[x]_+}-a|\le |x-a^2|/a
\)
for $a>0$,
\[
    |\hat q-q(R)|
    \lesssim
    \sqrt{\frac{t}{nR^3}}
    +
    \frac{\sqrt{dt}+t}{nR}.
\]
Both $R$ and $\hat R_{\rm DB}=q^{-1}(\hat q)$ belong to $[\hat r/2,2\hat r]\subset[R/4,3R]$. The mean value theorem and Lemma \ref{lem:qprime} therefore yield
\[
    |\hat R_{\rm DB}-R|
    \lesssim
    R^3|\hat q-q(R)|  \lesssim
    \sqrt{\frac{R^3t}{n}}
    +
    \frac{R^2(\sqrt{dt}+t)}{n}
    \lesssim
    \sqrt{\frac{R^3t}{n}},\]
where the last inequality follows from $n\gtrsim R(d+t)$. This proves (\ref{thm3:norm}).
\end{proof}

\section{Concluding Remarks}\label{sec:conclude}
In this paper, we study the finite-sample parameter estimation problem in logistic regression with Gaussian covariates. While this appears to be a basic enough model in statistics and machine learning, before this work, the minimax optimal estimation error rate has not been characterized. Also, the estimation performance of the MLE, one of the most canonical estimators, is not fully understood. We study the fundamental limit and develop optimal estimator for norm estimation. Combining with prior results on direction estimation, we show that the minimax optimal error rate is at the order $\Theta(\sqrt{Rd/n}+\sqrt{R^3/n})$, and this can be attained by an efficient algorithm. We also improve the best known estimation error rate of the MLE, $O(\sqrt{R^3d/n})$, to $\tilde{O}(\sqrt{Rd/n}+\sqrt{R^3/n}+R^2d/n)$, and conjecture that the extra term $R^2d/n$, while statistically suboptimal, may be intrinsic to the MLE. There remain some interesting problems for future research. First, as emphasized in Remark \ref{rem:mlenormopen}, it remains open whether $R^2d/n$ is a finite-sample lower bound on $\hat{R}_{\mle}-R$. If proven, this would lead to an interesting finding that the MLE is suboptimal in parameter estimation under $n\lesssim Rd\min\{d,R^2\}$. Moreover, 
it is also interesting to investigate to what extent our results extend to non-Gaussian designs. 
In fact, it was shown  that the upper bound of \cite{chardon2024finite} remains valid, up to some logarithmic factors, over a class of regular designs; see \cite[Theorem 3]{chardon2024finite}. We expect that our results also extend to a larger class of designs, while it remains rather subtle to identify these designs and extend some of our current arguments.

\subsection*{Acknowledgment}
Part of the work was done while JC was visiting Hal{\i}c{\i}o\u{g}lu Data Science Institute at UC San Diego in the summer of 2026, supported by NSF 2217058 and 2112665. 

\bibliography{libr}
\bibliographystyle{plain}

\appendix

\section{Gaussian Moment Bounds}\label{app:integral}
In this appendix, we collect the Gaussian moment bounds used in the proofs of Theorems \ref{thm:upper} and \ref{thm:minimaxupper}. 
We first note two elementary estimates. For $g\sim N(0,1)$, $R\ge1$, and every integer $p\ge0$,
\begin{align}
    \mathbb{E}|g|^p s'(Rg)
    &\le \frac{1}{\sqrt{2\pi}R^{p+1}}
    \int_{-\infty}^{\infty}|u|^ps'(u)\,du
    \le \frac{2p!}{\sqrt{2\pi}R^{p+1}},
    \label{T4bound1}\\
    \mathbb{E}|g|^p
    &\le (C\sqrt{p+1})^p.
    \label{eq:gaussian-moment-basic}
\end{align}
Here, we used $s'(u)\le e^{-|u|}$ and
$\int_{-\infty}^{\infty}|u|^pe^{-|u|}\,du=2p!$.

\begin{lem}[Moments under a nearby one-dimensional logistic model]
\label{lem:1dprobound}
Let $R\ge1$, and let $Z,T$ be independent $N(0,1)$ variables. Suppose
\(
    \alpha^2+\beta^2=1,~ \alpha\ge0,~
    0\le\beta\le \frac{c_0}{R}
\)
for a sufficiently small universal constant $c_0>0$. Conditioning on $(Z,T)$, let
\(
    Y\sim\Bern\big(s(R[\alpha Z+\beta T])\big).\) 
For $r\in[R/4,2R]$, define $\varepsilon_r=Y-s(rZ)$. Then, for every integer $p\ge0$,
\begin{align}
    \mathbb{E}\big[|Z|^p\varepsilon_r^2\big]
    &\le \frac{C^{p+1}p!}{R^{p+1}},
    \label{eq:nearby-radial-moment}\\
    \mathbb{E}\big[|T|^p\varepsilon_r^2\big]
    &\le \frac{C(C\sqrt{p+1})^p}{R}.
    \label{eq:nearby-transverse-moment}
\end{align}
Consequently,
\begin{equation}
    \mathbb{E}\varepsilon_r^2\le \frac CR,
    \label{eq:nearby-residual-energy}
\end{equation}
and, for every integer $p\ge2$,
\begin{align}
    \mathbb{E}\big|Z\varepsilon_r-\mathbb{E}(Z\varepsilon_r)\big|^p
    &\le \frac{C^pp!}{R^{p+1}},
    \label{eq:centered-radial-moment}\\
    \mathbb{E}\big|T\varepsilon_r-\mathbb{E}(T\varepsilon_r)\big|^p
    &\le \frac{C^pp!}{R}.
    \label{eq:centered-transverse-moment}
\end{align}
Moreover, for every integer $k\ge1$,
\begin{equation}
    \mathbb{E}\big|T\varepsilon_r-\mathbb{E}(T\varepsilon_r)\big|^{2k}
    \le \frac{C^kk!}{R}.
    \label{eq:centered-transverse-even}
\end{equation} 
\end{lem}

\begin{proof}
Put $U=\alpha Z+\beta T$. Conditioning on $(Z,T)$, using the randomness of $Y\sim\Bern(s(RU))$ gives
\begin{align}
     \mathbb{E}\big[\varepsilon_r^2\mid Z,T\big]= \big\{s(RU)-s(rZ)\big\}^2 + \mathbb{E}\big\{Y-s(RU)\big\}^2
   =\big\{s(RU)-s(rZ)\big\}^2+s'(RU).\label{condisecond}
\end{align}
Consider the event
\(\mathcal E=\{|Z|>4\beta|T|\}.\)
Since $\beta\le c_0/R\le c_0$, a sufficiently small $c_0$ ensures $\alpha\ge3/4$. On  the event $\mathcal E=\{|Z|>4\beta |T|\}$, $U$ and $Z$ have the same sign and $|U|\ge |Z|/2$. Also, $r\in[R/4,2R]$. In light of the inequalities $s(-|u|)\le e^{-|u|}$ for $u\le 0$ and $ s'(u)\le e^{-|u|}$ for $u\in \mathbb{R}$,
on the event $\mathcal E$, we use $|U|\ge |Z|/2$ to obtain \[s'(RU)\mathbf{1}(\mathcal{E})\le e^{-R|U|}\le e^{-R|Z|/2},\] and also combine with $\sign(U)=\sign(Z)$ to obtain 
\begin{align*}
    |s(RU)-s(rZ)| \mathbf{1}(\mathcal{E}) \le |s(-|RU|)-s(-|rZ|)|\le |s(-|RU|)|+|s(-|rZ|)|\le Ce^{-cR|Z|}.  
\end{align*}
On the event $\mathcal E^c=\{|Z|\le 4\beta|T|\}$, we use the trivial bound by a universal constant. Therefore, continuing from (\ref{condisecond}), 
\begin{align}
    \mathbb{E}\big[\varepsilon_r^2\mid Z,T\big]\le Ce^{-cR|Z|}
    +C\mathbf{1}\{|Z|\le4\beta|T|\}.
    \label{eq:residual-envelope}
\end{align}
\paragraph{Proof of (\ref{eq:nearby-radial-moment}).} We first prove \eqref{eq:nearby-radial-moment}. In light of \eqref{eq:residual-envelope},
\begin{align*}
    \mathbb{E}\big[|Z|^p\varepsilon_r^2\big]
    &\le C\mathbb{E}\big[|Z|^pe^{-cR|Z|}\big]
    +C\mathbb{E}\big[|Z|^p\mathbf{1}\{|Z|\le4\beta|T|\}\big].
\end{align*}
A change of variable, along with $\int_0^\infty u^pe^{-u}\,du=p!$, gives 
\[
    \mathbb{E}\big[|Z|^pe^{-cR|Z|}\big]
    \le \frac{2}{\sqrt{2\pi}}\int_0^\infty z^pe^{-cRz}\,dz
    =\frac{2}{\sqrt{2\pi}(cR)^{p+1}}\int_0^\infty u^pe^{-u}\,du
    \le \frac{C^{p+1}p!}{R^{p+1}}.
\]
For the second term, conditioning on $T$ and using that the density of $Z$ is uniformly bounded,
\begin{align*}
    \mathbb{E}\big[|Z|^p\mathbf{1}\{|Z|\le4\beta|T|\}\big]
    &\le C^{p+1}\beta^{p+1}\mathbb{E}|T|^{p+1}
    \le \frac{C^{p+1}p!}{R^{p+1}},
\end{align*}
where the last inequality follows from \eqref{eq:gaussian-moment-basic} and $\beta\le c_0/R$. This proves \eqref{eq:nearby-radial-moment}.

\paragraph{Proof of (\ref{eq:nearby-transverse-moment})--(\ref{eq:nearby-residual-energy}).}
Similarly, by independence of $Z$ and $T$,
\begin{align*}
    \mathbb{E}\big[|T|^p\varepsilon_r^2\big]
    &\le C\mathbb{E}|T|^p\,\mathbb{E}e^{-cR|Z|}
    +C\mathbb{E}\big[|T|^p\mathbf{1}\{|Z|\le4\beta|T|\}\big]\\
    &\le \frac{C\mathbb{E}|T|^p}{R}
    +C\beta\mathbb{E}|T|^{p+1}
    \le \frac{C(C\sqrt{p+1})^p}{R},
\end{align*}
where in the last inequality we use \eqref{eq:gaussian-moment-basic}, $\beta\le c_0/R$, and absorb the additional factor $\sqrt{p+1}$ into $C^p$.
This proves \eqref{eq:nearby-transverse-moment}; taking $p=0$ gives \eqref{eq:nearby-residual-energy}.

\paragraph{Proof of (\ref{eq:centered-radial-moment})--(\ref{eq:centered-transverse-even}).} It remains to establish the centered bounds. For any random variable $X$ and integer $p\ge1$, Jensen's inequality gives
\(
    \mathbb{E}|X-\mathbb{E}X|^p\le 2^p\mathbb{E}|X|^p.\)
Since $|\varepsilon_r|\le1$, for $p\ge2$ we have
\[|Z\varepsilon_r|^p\le |Z|^p\varepsilon_r^2,\qquad |T\varepsilon_r|^p\le |T|^p\varepsilon_r^2.\] Combining these facts with \eqref{eq:nearby-radial-moment}--\eqref{eq:nearby-transverse-moment}, together with $(C\sqrt{p+1})^p\le C_1^pp!$, yields \eqref{eq:centered-radial-moment}--\eqref{eq:centered-transverse-moment}. Taking $p=2k$ in \eqref{eq:nearby-transverse-moment}, and using $(C\sqrt{2k+1})^{2k}\le C_2^kk!$, gives \eqref{eq:centered-transverse-even}. 
\end{proof}

\begin{lem}[Moments of localized directional increments]
\label{lem:directincrebound}
Let $r\in[R/2,2R]$ and $\bv,\bv^*\in\mathbb{S}^{d-1}$ be fixed. Define
\[
    \Delta_i(r,\bv)=\psi_i(r,\bv)-\psi_i(r,\bv^*),
    \qquad
    \psi_i(r,\bv)=\{s(r\bx_i^\top\bv)-y_i\}\bx_i^\top\bv,
\]
where $\bx_i\sim N(\bm{0},\bI_d)$,
$y_i\mid\bx_i\sim\Bern(s(R\bx_i^\top\bv^*))$, and
$\eta=\|\bv-\bv^*\|_2$. There exists a universal constant $C>0$ such that, if $R\eta\le c_0$ for a sufficiently small universal constant $c_0>0$, then for every integer $p\ge2$,
\[
    \mathbb{E}|\Delta_i(r,\bv)|^p
    \le C^pp!\frac{\eta^p}{R}.
\]
\end{lem}

\begin{proof}
If $\eta=0$, then $\Delta_i(r,\bv)=0$ and we are done. We thus consider $\eta>0$ and write
\[
    \ba=\frac{\bv-\bv^*}{\eta},
    \qquad
    \bv_\tau=\bv^*+\tau\eta\ba,
    \qquad
    Z_\tau=\bx_i^\top\bv_\tau,
    \qquad
    W=\bx_i^\top\ba,
    \qquad \tau\in[0,1].
\]
For fixed $y\in\{0,1\}$, define
$f_y(z)=z\{s(rz)-y\}$. Then
$f_y'(z)=s(rz)-y+rzs'(rz)$, and the fundamental theorem of calculus gives
\begin{align}
    \Delta_i(r,\bv)
    &=\eta W\int_0^1 A_\tau\,d\tau,
    \qquad
    A_\tau=s(rZ_\tau)-y_i+rZ_\tau s'(rZ_\tau).
    \label{eq:direction-increment-ftc}
\end{align}
Since $|A_\tau|\le C$, Jensen's inequality yields
\begin{equation}
    \mathbb{E}|\Delta_i(r,\bv)|^p\le \int_0^1
    \mathbb{E}\big[|\eta W|^pA_\tau^p\big]d\tau
    \le (C\eta)^p\int_0^1
    \mathbb{E}\big[|W|^pA_\tau^2\big]d\tau.
    \label{eq:direction-increment-reduction}
\end{equation}
It remains to show that the integrand is at most $C^pp!/R$, uniformly over $\tau\in[0,1]$.

Fix $\tau\in[0,1]$. Let
\[
    \sigma_\tau=\|\bv_\tau\|_2,
    \qquad
    \bar{\bv}_\tau=\frac{\bv_\tau}{\sigma_\tau}.
\]
Since $\bv$ and $\bv^*$ are unit vectors, some algebra finds
\[
    \sigma_\tau^2=1-\tau(1-\tau)\eta^2,
\]
and hence $\sigma_\tau\asymp1$. In the two-dimensional subspace spanned by $\bv$ and $\bv^*$, choose a unit vector $\bu_\tau\perp\bar{\bv}_\tau$ and write
\[
    \bv^*=\alpha_\tau\bar{\bv}_\tau+\beta_\tau\bu_\tau,
    \qquad
    \ba=m_\tau\bar{\bv}_\tau+n_\tau\bu_\tau.
\]
Here $\alpha_\tau\ge0$, $|\beta_\tau|\le C\eta\le c_0/R$, and
$|m_\tau|+|n_\tau|\le C$. Replacing $\bu_\tau$ by $-\bu_\tau$ if necessary, we may assume $\beta_\tau\ge0$. Let
\[
    Z=\bx_i^\top\bar{\bv}_\tau,
    \qquad
    T=\bx_i^\top\bu_\tau.
\]
Then $Z,T$ are independent standard normal variables,
\[
    Z_\tau=\sigma_\tau Z,
    \qquad
    W=m_\tau Z+n_\tau T,
\]
and, conditioning on $(Z,T)$,
\[
    y_i\sim\Bern\big(s(R[\alpha_\tau Z+\beta_\tau T])\big).
\]
Moreover, $r_\tau:=r\sigma_\tau\in[R/4,2R]$. Therefore, Lemma \ref{lem:1dprobound} and the inequality
$|m_\tau Z+n_\tau T|^p\le C^p(|Z|^p+|T|^p)$ give
\begin{equation}
    \mathbb{E}\Big[|W|^p\{s(rZ_\tau)-y_i\}^2\Big]
    \le \frac{C^pp!}{R}.
    \label{eq:direction-residual-moment}
\end{equation}

We next control the derivative term. By independence of $Z$ and $T$, \eqref{eq:gaussian-moment-basic}, and $s'(u)^2\le e^{-|u|}$,
\begin{align*}
     \mathbb{E}\Big[|W|^p(rZ_\tau)^2s'(rZ_\tau)^2\Big] 
    & \le C^pr_\tau^2\mathbb{E}\big[|Z|^{p+2}s'(r_\tau Z)^2\big]
    +C^pr_\tau^2\mathbb{E}|T|^p\,
    \mathbb{E}\big[Z^2s'(r_\tau Z)^2\big]\\
    &\le \frac{C^p(p+2)!}{r_\tau^{p+1}}
    +\frac{(C\sqrt{p+1})^p}{r_\tau}
    \le \frac{C^pp!}{R}.
\end{align*}
Together with \[A_\tau^2
    \le2\{s(rZ_\tau)-y_i\}^2
    +2(rZ_\tau)^2s'(rZ_\tau)^2,\] 
    this and \eqref{eq:direction-residual-moment} imply
\(
    \mathbb{E}\big[|W|^pA_\tau^2\big]
    \le \frac{C^pp!}{R}\) 
uniformly over $\tau\in[0,1]$. Substituting this into \eqref{eq:direction-increment-reduction} proves the desired claim.
\end{proof}

\section{Covering Arguments in Proofs of Lemmas \ref{lem:true-direction-process}--\ref{lem:localized-increment}}\label{app:covering}
\subsection{Covering Argument for Lemma \ref{lem:true-direction-process}}\label{app:lem53covering}
We now make the bound uniform over \(r\in[R/2,2R]\). For convenience we define $\Psi_i(r):=\big(s(rg_i)-y_i\big)g_i-\mathbb{E}\big[\big(s(rg_i)-s(Rg_i)\big)g_i\big]$, and hence we can write $S_n(r,\bv^*)=\frac{1}{n}\sum_{i=1}^n\Psi_i(r) $ in view of (\ref{Snexpandver}). Recall that $g_i=\bx_i^\top\bv^*\sim N(0,1)$. By Lemma \ref{lem:maxgaussian}, 
\[
    \mathbb{P}\bigg(\max_{i\in[n]}|g_i|\le 2\sqrt{\log n}\bigg)\ge 1- \frac{2}{n}.
\]
On this event, in light of 
\[\bigg|\frac{d}{dr}[(s(rg_i)-y_i)g_i]\bigg|=g_i^2s'(rg_i) \lesssim \log n, \qquad \forall i\in[n],\] 
we have that $(s(rg_i)-y_i)g_i$ is $O(\log n)$-Lipschitz for all $i\in[n]$. Moreover, in light of 
\[\frac{d}{dr}\mathbb{E}\big[\big(s(rg_i)-s(Rg_i)\big)g_i\big] = \mathbb{E}\big[s'(rg_i)g_i^2\big]=O(1),\] it follows that 
$\mathbb{E}\big[\big(s(rg_i)-s(Rg_i)\big)g_i\big]$ is $O(1)$-Lipschitz. Taken collectively, we know that 
$\Psi_i(r)$ is $O(\log n)$-Lipschitz, and hence $S_n(r,\bv^*)$ is $O(\log n)$-Lipschitz on $r$. For some small enough $\varepsilon$ to be chosen, take an \(\varepsilon\)-net of \([R/2,2R]\), denoted by $N_{\varepsilon}$, and suppose that  $|N_{\varepsilon}|\le\frac{2R}{\varepsilon}$. For any $r\in [\frac{R}{2},2R]$, we let $r_{\varepsilon}:=\textrm{arg}\min_{w\in N_\varepsilon}|w-r|$ (with ties broken arbitrarily). Then, 
\begin{align*}\sup_{r\in[R/2,2R]}|S_n(r,\bv^*)|&\le\sup_{r\in[R/2,2R]}|S_n(r_{\varepsilon},\bv^*)|+\sup_{r\in[R/2,2R]}|S_n(r,\bv^*)-S_n(r_{\varepsilon},\bv^*)|  \nn\\
    &\stackrel{(a)}{\le} \sup_{r\in N_\varepsilon}|S_n(r,\bv^*)|+ O(\varepsilon\log n) \\ 
    &\stackrel{(b)}{\lesssim} \sqrt{\frac{\log(R/\varepsilon)}{nR^3}} + \frac{
    \log(R/\varepsilon)
    }{nR} + \varepsilon\log n,
\end{align*}
where $(a)$ holds because $r_\varepsilon\in N_\varepsilon$ and $S_n(r,\bv^*)$ is $O(\log n)$-Lipschitz on $r$, $(b)$ holds with probability at least $1-\frac{\varepsilon}{R}$ by the non-uniform bound (\ref{nonuniform1}) along with a union bound over $N_\varepsilon$ and a suitable choice of $t$. We now choose $\varepsilon = \frac{1}{nR}$, then the desired bound holds with at least $1-\frac{C}{n}$ probability. The proof is complete.

\subsection{Covering Argument for Lemma \ref{lem:localized-increment}}\label{app:lem54covering}

We now invoke a covering argument to make the bound uniform. Let $U=d\log(nR).$
Define
$
    \Gamma
    =
    C\big\{
        \delta\sqrt{\frac{U}{nR}}
        +
        \delta\frac{U}{n}\big\}.$ 
Take an \(\varepsilon_v\)-net \(\mathcal N_v\) of \(\mathcal V_\delta\) with
$
    \varepsilon_v=\frac{\Gamma}{C\sqrt d},$ 
and an \(\varepsilon_r\)-net \(\mathcal N_r\) of \([R/2,2R]\) with
$\varepsilon_r=\frac{\Gamma}{C}.$ By standard covering number bound we suppose $
    |\mathcal N_v|
    \le
    \left(\frac{C\delta}{\varepsilon_v}\right)^d,~~|\mathcal N_r|
    \le
    \frac{CR}{\varepsilon_r}.$
Increasing constants if necessary, we have 
$\log|\mathcal N_v|+\log|\mathcal N_r|\le C d\log(nR)$. Applying (\ref{fixedrvbound}) with \(t=C d\log(nR)\) and taking
a union bound over \(\mathcal N_r\times\mathcal N_v\), we get
\begin{align}
     \mathbb{P}\bigg(\max_{r\in\mathcal N_r}
    \max_{\bv\in\mathcal N_v}
    |S_n(r,\bv)-S_n(r,\bv^*)|
    \le
    C\bigg\{
        \delta\sqrt{\frac{d\log(nR)}{nR}}
        +
        \delta\frac{d\log(nR)}{n}
    \bigg\}\bigg)\ge 1-2e^{-cd}.\label{netbound11}
\end{align} 
By Lemma \ref{lem:operconcen},
\begin{align}
    \bigg\|
        \frac1n\sum_{i=1}^n \bx_i \bx_i^\top
    \bigg\|_{\mathrm{op}}\le C,\qquad \frac1n\sum_{i=1}^n \|\bx_i\|_2\le C\sqrt d\label{2events1}
\end{align}
hold with probability at least \(1-2e^{-d}\). On these events, we claim that the map
$
    (r,\bv)\mapsto S_n(r,\bv)-S_n(r,\bv^*)$ 
is \(O(1)\)-Lipschitz in \(r\) and \(O(\sqrt d)\)-Lipschitz in \(\bv\). We pause to show this fact. Recall that $S_n(r,\bv)-S_n(r,\bv^*)$ can be written as (\ref{differiidsum}), where $\psi_i(r,\bv)=\{s(r\bx_i^\top\bv)-y_i\}\bx_i^\top\bv$, 
and therefore
\[
    S_n(r,\bv)-S_n(r,\bv^*) = \frac{1}{n}\sum_{i=1}^n\big[\psi_i(r,\bv)-\mathbb{E}\psi_i(r,\bv)\big] -\frac{1}{n}\sum_{i=1}^n\big[\psi_i(r,\bv^*)-\mathbb{E}\psi_i(r,\bv^*)\big].  
\]

We first show that the map
$
    (r,\bv)\mapsto S_n(r,\bv)-S_n(r,\bv^*)$ 
is \(O(1)\)-Lipschitz in \(r\). Note that 
$\frac{\partial}{\partial r}\psi_i(r,\bv)=(\bx_i^\top\bv)^2s'(r\bx_i^\top\bv),$ and hence 
\[
    \frac{\partial}{\partial r}\frac{1}{n}\sum_{i=1}^n\big[\psi_i(r,\bv)-\mathbb{E}\psi_i(r,\bv)\big]= \frac{1}{n}\sum_{i=1}^n (\bx_i^\top\bv)^2s'(r\bx_i^\top\bv)-\mathbb{E}\big[(\bx_i^\top\bv)^2s'(r\bx_i^\top\bv)\big]. 
\]
By $s'(a)\le 1/4$ and the first event in (\ref{2events1}),
\[
    \bigg| \frac{\partial}{\partial r}\frac{1}{n}\sum_{i=1}^n\big[\psi_i(r,\bv)-\mathbb{E}\psi_i(r,\bv)\big]\bigg| \le \frac{1}{4n}\sum_{i=1}^n (\bx_i^\top\bv)^2 + \frac{1}{4}\mathbb{E}(\bx_i^\top\bv)^2 =O(1)
\]
By treating $\frac{1}{n}\sum_{i=1}^n\big[\psi_i(r,\bv^*)-\mathbb{E}\psi_i(r,\bv^*)\big]$ similarly, we arrive at the $O(1)$-Lipschitzness in $r$.

We then show that the map
$
    (r,\bv)\mapsto S_n(r,\bv)-S_n(r,\bv^*)$ 
is \(O(\sqrt{d})\)-Lipschitz in \(\bv\).
We compute the gradient 
\[
    \nabla_\bv\psi_i(r,\bv)
    =
    \bx_i
    \left[
        s(r\bx_i^\top\bv)-y_i
        +
        r\bx_i^\top\bv s'(r\bx_i^\top\bv)
    \right]:=\bx_i \phi(r,\bv),
\] 
which then yields
\[
    \nabla_\bv \{S_n(r,\bv)-S_n(r,\bv^*)\} = \frac{1}{n}\sum_{i=1}^n     
   \phi(r,\bv) \bx_i - \mathbb{E}\big[\phi(r,\bv)\bx_i\big].
\]
Note that $\phi(r,\bv)$ is uniformly bounded because 
\(
    |s(r\bx_i^\top\bv)-y_i|\le 1
\)
and
\(
    \sup_{u\in\mathbb R}|u s'(u)|=O(1).\)
Thus, on the second event of (\ref{2events1}), 
\[
    \|\nabla_\bv \{S_n(r,\bv)-S_n(r,\bv^*)\}\|_2\lesssim \frac{1}{n}\sum_{i=1}^n\|\bx_i\|_2 +\mathbb{E}\|\bx_i\|_2\lesssim \sqrt{d}, 
\]
as claimed.

We are ready to conclude the proof. By the choices of \(\varepsilon_r\) and \(\varepsilon_v\), the discretization
error is at most \(\Theta(\Gamma)=\Theta(\delta\sqrt{\frac{U}{nR}}+\delta\frac{U}{n})\), which does not increase the order of the bound in (\ref{netbound11}). This proves the desired claim. 

\section{Deferred Details in Proof of Lemma \ref{lem:direct-U}}\label{app:bernsteinminimax}
\subsection{Bounding $L_m$}

We first control the centered variables
\[
    \eta_i=e_iT_i-\mathbb{E}(e_iT_i)
    =e_iT_i-q(R)\beta.
\]
By the centered transverse moment bounds in Lemma \ref{lem:1dprobound}, for every integer $p\ge2$ and $k\ge1$,
\begin{align}
  \mathbb{E}|\eta_i|^p
  &\le\frac{C^pp!}{R},
  \label{eq:eta-q-moments}\\
  \mathbb{E}|\eta_i|^{2k}
  &\le\frac{C^kk!}{R}.
  \label{eq:eta-even}
\end{align}
Therefore, Lemma~\ref{lem:bernstein210} and $m\ge CRt$ give, with probability at least $1-2e^{-t}$,
\begin{equation}
    \bigg|\sum_{i=1}^m\eta_i\bigg|
    \lesssim \sqrt{\frac{mt}{R}}+t
    \lesssim \sqrt{\frac{mt}{R}}.
    \label{meanetaibound}
\end{equation}
Substituting \eqref{meanetaibound} into \eqref{eq:L-exact} and using the boundedness of $q(R)$,
\begin{equation}
  |L_m|
  \lesssim\beta\sqrt{\frac{t}{mR}}.
  \label{eq:L-pre}
\end{equation}
Further using
$\beta=\sqrt{1-(\breve{\bv}^\top\bv^*)^2}
\le \sqrt{\frac{C(d+t)}{mR}}$, we arrive at (\ref{eq:L-final}).
\subsection{Bounding $Q_m^{\rm sc}$}
Next, we proceed to bound
\begin{equation}
Q_m^{\rm sc} = \frac{1}{m(m-1)}\sum_{i\ne j}\eta_i\eta_j = \frac{1}{m(m-1)}\bigg(\sum_{i=1}^m\eta_i\bigg)^2
   -\frac{1}{m(m-1)}\sum_{i=1}^m\eta_i^2.
  \label{eq:scalar-identity}
\end{equation}
To control the second term, let
$V_i=\eta_i^2-\mathbb{E}\eta_i^2$. By \eqref{eq:eta-even}, for every integer $k\ge2$,
\[
    \mathbb{E}|V_i|^k
    \le 2^{k-1}\big\{\mathbb{E}|\eta_i|^{2k}+
    (\mathbb{E}\eta_i^2)^k\big\}
    \le \frac{C^kk!}{R}.
\]
Another application of Lemma~\ref{lem:bernstein210} gives
\begin{equation}
  \bigg|\sum_{i=1}^m(\eta_i^2-\mathbb{E}\eta_i^2)\bigg|
  \le C\bigg\{\sqrt{\frac{mt}{R}}+t\bigg\}
  \label{eq:Vi-sum}
\end{equation}
with probability at least $1-2e^{-t}$. Since
$\mathbb{E}\eta_i^2\le C/R$, \eqref{eq:Vi-sum} and $m\ge CRt$ imply
\begin{equation}
  \sum_{i=1}^m\eta_i^2
  \lesssim \frac mR+\sqrt{\frac{mt}{R}}+t
  \lesssim\frac mR.
  \label{eq:eta-square-sum}
\end{equation}
Substituting \eqref{eq:eta-square-sum} and \eqref{meanetaibound} into
\eqref{eq:scalar-identity} yields (\ref{eq:Qsc-final}).
\section{Proof of Proposition \ref{prop:asymptotic-radial-bias}}
\label{sec:small-aspect-ratio-bias}
This section derives Proposition \ref{prop:asymptotic-radial-bias} from the high-dimensional asymptotic theory of Zhao, Sur, and Cand\`es \cite{zhao2022asymptotic}. We make the asymptotic setting explicit. Fix $R>0$. For each $n$, let $d=d_n$ and let $\btheta_n^*\in\mathbb{R}^{d_n}$ satisfy $\|\btheta_n^*\|_2=R$. For each fixed aspect ratio $\delta>0$, we first let $n,d_n\to\infty$ with $d_n/n\to\delta$, and then study the resulting deterministic limits as $\delta\to0^+$. We suppress the dependence on $n$ when there is no ambiguity. We also note that \cite{zhao2022asymptotic} uses responses in $\{-1,1\}$. Under the transformation $\tilde y_i=2y_i-1$, their logistic loss is identical to (\ref{mlelogi}) up to a multiplicative constant, and hence their MLE and state equations apply to our formulation.

Recall that
\(
q(r):=\mathbb{E}_{g\sim N(0,1)}[g s(rg)],\) and for convenience, define
\begin{equation}
    \mu(r)
    :=
    \mathbb{E}_{g\sim N(0,1)}[s'(rg)],
    \qquad
    \nu(r)
    :=
    \mathbb{E}_{g\sim N(0,1)}[g^2s'(rg)]
    =
    q'(r).
    \label{eq:app-mu-nu}
\end{equation}
By Stein's identity,
\begin{equation}
    q(r)
    =
    \mathbb{E}_{g\sim N(0,1)}[g s(rg)]
    =
    r\mathbb{E}_{g\sim N(0,1)}[s'(rg)]
    =
    r\mu(r).
    \label{eq:app-q-mu}
\end{equation}
Differentiating (\ref{eq:app-q-mu}) gives
\begin{equation}
    \nu(r)
    =
    \mu(r)+r\mu'(r),
    \qquad
    \mu'(r)
    =
    \mathbb{E}_{g\sim N(0,1)}[g s''(rg)].
    \label{eq:app-nu-mu}
\end{equation}
Since $Gs''(RG)<0$ almost surely away from $G=0$, we have $0<\nu(R)<\mu(R)$. In particular, the asymptotically precise pre-constant in Proposition \ref{prop:asymptotic-radial-bias} is
\[
    \mathfrak{c}_R
    :=
    \frac{1}{2R\nu(R)}
    =
    \frac{1}{2Rq'(R)}.
\]

Let $\bv_n^*:=\btheta_n^*/R$ and
\[
    \bP_n^\perp
    :=
    \bI_{d_n}-\bv_n^*(\bv_n^*)^\top.
\]
For the ordinary MLE $\hat{\btheta}_n$, define the signal-aligned coefficient and orthogonal energy by
\begin{equation}
    M_n
    :=
    \frac{\langle\hat{\btheta}_n,\btheta_n^*\rangle}{R^2},
    \qquad
    V_n^2
    :=
    \|\bP_n^\perp\hat{\btheta}_n\|_2^2.
    \label{eq:app-finite-overlap-energy}
\end{equation}
By orthogonality,
\begin{equation}
    \|\hat{\btheta}_n\|_2^2
    =
    M_n^2R^2+V_n^2.
    \label{eq:app-exact-norm-decomposition}
\end{equation}
The following result specializes \cite[Lemma~2.1, Lemma~3.1, and Eq.~(3.4)]{zhao2022asymptotic} to identity covariance. In their notation,
\[
    \alpha(n)
    =
    \frac{\langle\hat{\btheta}_n,\btheta_n^*\rangle}
         {\|\btheta_n^*\|_2^2}
    =
    M_n,
    \qquad
    \sigma(n)^2
    =
    \|\bP_n^\perp\hat{\btheta}_n\|_2^2
    =
    V_n^2.
\]
We denote their limiting parameters $(\alpha_\star,\sigma_\star,\lambda_\star)$ by $(a_\delta,\tau_\delta,\lambda_\delta)$, respectively.

\begin{lem}[High-dimensional asymptotics, \cite{zhao2022asymptotic}] 
    \label{lem:app-bej-specialization}
    Fix \(R>0\) and \(\delta>0\), and suppose that
    \((\delta,R)\) lies in the asymptotic MLE-existence region.
    If \(n,d_n\to\infty\) with \(d_n/n\to\delta\), then there exist
    deterministic positive constants $(a_{\delta},\tau_{\delta},\lambda_{\delta})$
    depending only on \((\delta,R)\), such that
    \begin{equation}
        M_n\longrightarrow a_{\delta},
        \qquad
        V_n^2\longrightarrow
        \delta\tau_{\delta}^2
        \label{eq:app-bej-limits}
    \end{equation}
    almost surely. Moreover, the triple
    \((a_{\delta},\tau_{\delta},\lambda_{\delta})\)
    is the unique solution of the following state equations (\ref{eq:app-bej-state-1})--(\ref{eq:app-bej-state-3}).
    Let \(G,W\) be independent $N(0,1)$ variables and set
    \[
        Q:=RG,
        \qquad 
        Q_{\delta,2}
        :=
        -a_{\delta}Q
        +
        \sqrt{\delta}\,\tau_{\delta}W.\] 
    Define $\rho(t):=\log(1+\exp(t))$ 
    and
    \[\operatorname{prox}_{\lambda\rho}(z)
        :=
        \arg\min_{x\in\mathbb{R}}
        \bigg\{
            \lambda\rho(x)
            +
            \frac{1}{2}(x-z)^2
        \bigg\}.\]  
    Then
    \begin{align}
        \tau_{\delta}^2
        &=
        \frac{1}{\delta^2}
        \mathbb{E}\left[
            2s(Q)
            \left\{
                \lambda_{\delta}
                s\!\left(
                    \operatorname{prox}_{\lambda_{\delta}\rho}
                    (Q_{\delta,2})
                \right)
            \right\}^2
        \right],
        \label{eq:app-bej-state-1}
        \\
        0
        &=
        \mathbb{E}\left[
            s(Q)Q
            \lambda_{\delta}
            s\!\left(
                \operatorname{prox}_{\lambda_{\delta}\rho}
                (Q_{\delta,2})
            \right)
        \right],
        \label{eq:app-bej-state-2}
        \\
        1-\delta
        &=
        \mathbb{E}\left[
            \frac{
                2s(Q)
            }{
                1+
                \lambda_{\delta}
                s'\!\left(
                    \operatorname{prox}_{\lambda_{\delta}\rho}
                    (Q_{\delta,2})
                \right)
            }
        \right].
        \label{eq:app-bej-state-3}
    \end{align}
\end{lem}

By \eqref{eq:app-exact-norm-decomposition},
Lemma~\ref{lem:app-bej-specialization}, and the continuous mapping
theorem,
\begin{equation}
    \hat R_{\mathrm{MLE}}
    :=
    \|
        \hat{\btheta}_n
    \|_2
    \longrightarrow
    r_{\mathrm{MLE}}(\delta,R)
    :=
    \sqrt{
        a_{\delta}^2R^2
        +
        \delta\tau_{\delta}^2
    }
    \label{eq:app-mle-limit-radius}
\end{equation}
almost surely. It remains to expand the deterministic function
\(r_{\mathrm{MLE}}(\delta,R)\) as \(\delta\to 0^+\).

\subsection{Change of Variables}

Introduce
\[
    u_{\delta}:=\tau_{\delta}^2,
    \qquad
    \ell_{\delta}:=\frac{\lambda_{\delta}}{\delta}.
\]
Then
\[
    r_{\mle}(\delta,R)
    =
    \sqrt{a_\delta^2R^2+\delta u_\delta}.
\]
For \(\delta>0\), define the shorthand
\begin{equation}
    X_{\delta}
    :=
    \operatorname{prox}_{\delta\ell_{\delta}\rho}
    (
        -a_{\delta}Q
        +
        \sqrt{\delta u_{\delta}}\,W
    ).
    \label{eq:app-Xdelta}
\end{equation}
Substituting
\(\lambda_{\delta}=\delta\ell_{\delta}\)
into \eqref{eq:app-bej-state-1}--\eqref{eq:app-bej-state-3},
we obtain 
\begin{align}
    u_{\delta}
    &=
    2\ell_{\delta}^2
    \mathbb{E}\bigl[
        s(Q)s(X_{\delta})^2
    \bigr],
    \label{eq:app-scaled-state-1}
    \\
    0
    &=
    \mathbb{E}\bigl[
        s(Q)Q s(X_{\delta})
    \bigr],
    \label{eq:app-scaled-state-2}
    \\
    1
    &=
    2\ell_{\delta}
    \mathbb{E}\left[
        \frac{
            s(Q)s'(X_{\delta})
        }{
            1+
            \delta\ell_{\delta}s'(X_{\delta})
        }
    \right].
    \label{eq:app-scaled-state-3}
\end{align}
To see \eqref{eq:app-scaled-state-2}, note that the second state equation in \eqref{eq:app-bej-state-2}
becomes 
\[
    0
    =
    \delta\ell_\delta
    \mathbb{E}\big[s(Q)Qs(X_\delta)\big].
\]
Since $\lambda_\delta=\delta\ell_\delta>0$ for every $\delta>0$,
dividing both sides by $\delta\ell_\delta$ yields \eqref{eq:app-scaled-state-2}.
To verify \eqref{eq:app-scaled-state-3}, note first that the symmetry of
\(Q\) gives
$
    \mathbb{E}[2s(Q)]=1.$ 
Hence \eqref{eq:app-bej-state-3} implies
\[
\begin{aligned}
    \delta
    =
    1-
    \mathbb{E}\left[
        \frac{
            2s(Q)
        }{
            1+
            \delta\ell_{\delta}s'(X_{\delta})
        }
    \right]
    =
    \delta
    \mathbb{E}\left[
        \frac{
            2s(Q)\ell_{\delta}s'(X_{\delta})
        }{
            1+
            \delta\ell_{\delta}s'(X_{\delta})
        }
    \right],
\end{aligned}
\]
and division by \(\delta\) yields
\eqref{eq:app-scaled-state-3}. In the following, we work with the more amenable system of equations (\ref{eq:app-scaled-state-1})--(\ref{eq:app-scaled-state-3}). 

\subsection{Near-$0$ Behaviors of Solutions to (\ref{eq:app-scaled-state-1})--(\ref{eq:app-scaled-state-3})} 
\begin{lem}
    \label{lem:app-smooth-branch}
    For every fixed \(R>0\), the solution of
    \eqref{eq:app-scaled-state-1}--\eqref{eq:app-scaled-state-3}
    has a smooth local extension to \(\delta=0\). Under the parametrization $t=\sqrt{\delta}$, this local extension is
even in $t$. Consequently, for some constant $a_1(R)$ depending only
on $R$,
\[
    a_\delta
    =
    1+a_1(R)\delta+O_R(\delta^2),
    \qquad
    u_\delta
    =
    \frac{1}{\mu(R)}+O_R(\delta),
    \qquad
    \ell_\delta
    =
    \frac{1}{\mu(R)}+O_R(\delta).
\] 
\end{lem}

\begin{proof}
    Set
    $
        t:=\sqrt{\delta}.$ 
    For \(u>0\), \(\ell>0\), and \(a>0\), define
    \begin{align}\label{eq:prox-foc}
        \mathcal{X}(t;u,\ell,a)
        :=
        \operatorname{prox}_{t^2\ell\rho}
        \left(
            -aQ+t\sqrt{u}\,W
        \right).
    \end{align}
    Equivalently,
    \(x=\mathcal{X}(t;u,\ell,a)\)
    is the unique solution of 
    \begin{equation}
        x+t^2\ell s(x)
        =
        -aQ+t\sqrt{u}\,W.
        \label{eq:app-prox-implicit}
    \end{equation}
    The derivative of the left-hand side with respect to \(x\) is
    \(
        1+t^2\ell s'(x)\ge1.\)
    Therefore, \eqref{eq:app-prox-implicit} has a unique solution, and the scalar implicit function theorem shows that \(\mathcal{X}\) is smooth in \((t,u,\ell,a)\) in a neighborhood of
    \[
        \left(
            0,
            \frac{1}{\mu(R)},
            \frac{1}{\mu(R)},
            1
        \right).
    \]
    Fix a compact neighborhood of this point on which \(u,\ell,a\) are bounded away from zero. Repeated implicit differentiation of \eqref{eq:app-prox-implicit} shows that, for every fixed multi-index \(\gamma\), there exist constants \(C_\gamma,k_\gamma>0\) such that
    \[
        \big|
            \partial^\gamma\mathcal{X}(t;u,\ell,a)
        \big|
        \le
        C_\gamma(1+|Q|+|W|)^{k_\gamma}
    \]
    uniformly on this neighborhood. Indeed, every derivative is a finite sum of products of lower-order derivatives and bounded derivatives of \(s\), divided by a power of \(1+t^2\ell s'(\mathcal{X})\ge1\). Gaussian integrability therefore justifies all differentiations under the expectations below. Define  
    \begin{align*}
        F_1(t,u,\ell,a)
        &:=
        u
        -
        2\ell^2
        \mathbb{E}\left[
            s(Q)
            s\!\left(
                \mathcal{X}(t;u,\ell,a)
            \right)^2
        \right],
        \\
        F_2(t,u,\ell,a)
        &:=
        \mathbb{E}\left[
            s(Q)Q
            s\!\left(
                \mathcal{X}(t;u,\ell,a)
            \right)
        \right],
        \\
        F_3(t,u,\ell,a)
        &:=
        1
        -
        2\ell
        \mathbb{E}\left[
            \frac{
                s(Q)
                s'\!\left(
                    \mathcal{X}(t;u,\ell,a)
                \right)
            }{
                1+
                t^2\ell
                s'\!\left(
                    \mathcal{X}(t;u,\ell,a)
                \right)
            }
        \right].
    \end{align*}
    Let \(F=(F_1,F_2,F_3)\). At \(t=0\), by definition we have 
    \[
        \mathcal{X}(0;u,\ell,a)=-aQ.
    \]

    We note the following identities: 
    \begin{equation}
        \mathbb{E}\bigl[
            s(Q)s(-Q)^2
        \bigr]
        =
        \frac{\mu(R)}{2},
        \qquad
        \mathbb{E}\bigl[
            s(Q)s'(Q)
        \bigr]
        =
        \frac{\mu(R)}{2},
        \qquad
        \mathbb{E}\bigl[
            Qs'(Q)
        \bigr]
        =
        0.
        \label{eq:app-basic-symmetry}
    \end{equation}
    We pause to prove them. For the first identity, symmetry of \(Q\), together with
\(s(-x)=1-s(x)\) and \(s'(x)=s(x)s(-x)\), gives
\begin{align*}
    2\mathbb{E}\bigl[s(Q)s(-Q)^2\bigr]
    &=
    \mathbb{E}\bigl[
        s(Q)s(-Q)^2+s(-Q)s(Q)^2
    \bigr] \\
    &=
    \mathbb{E}\bigl[
        s(Q)s(-Q)\{s(Q)+s(-Q)\}
    \bigr] =
    \mathbb{E}[s'(Q)]
    =
    \mu(R).
\end{align*}
For the second identity, since \(s'\) is even,
\begin{align*}
    2\mathbb{E}[s(Q)s'(Q)]
    &=
    \mathbb{E}\bigl[
        \{s(Q)+s(-Q)\}s'(Q)
    \bigr] =
    \mathbb{E}[s'(Q)]
    =
    \mu(R).
\end{align*}
Finally, \(Qs'(Q)\) is odd, and therefore
\(
    \mathbb{E}[Qs'(Q)]=0.
\)

Consequently, we have
    \begin{equation}
        F_i\left(
            0,
            \frac{1}{\mu(R)},
            \frac{1}{\mu(R)},
            1
        \right)
        =
        0,\quad 1\le i\le 3,
        \label{eq:app-base-solution}
    \end{equation}
in light of 
\begin{gather*}
      F_1\left(
            0,
            \frac{1}{\mu(R)},
            \frac{1}{\mu(R)},
            1
        \right) = \frac{1}{\mu(R)}- \frac{2\mathbb{E}[s(Q)s(-Q)^2]}{(\mu(R))^2}=0;
        \\
        F_2\left(
            0,
            \frac{1}{\mu(R)},
            \frac{1}{\mu(R)},
            1
        \right) = \mathbb{E}\big[Qs(Q)s(-Q)\big] =  \mathbb{E}\big[Qs'(Q)\big] =0;\\
         F_3\left(
            0,
            \frac{1}{\mu(R)},
            \frac{1}{\mu(R)},
            1
        \right) = 1-\frac{2\mathbb{E}[s(Q)s'(-Q)]}{\mu(R)} =0 . 
\end{gather*}

    We next compute the Jacobian with respect to
    \((u,\ell,a)\)
    at the point in \eqref{eq:app-base-solution}.
  We repeatedly use the following consequence of symmetry: if \(h\) is
even and integrable, then
\[
    \mathbb{E}[s(Q)h(Q)]
    =  \mathbb{E}\bigg[s(Q)\frac{h(Q)+h(-Q)}{2}\bigg]= \frac{1}{2}\mathbb{E}\bigg[s(Q)h(Q)+s(-Q)h(Q)\bigg]=
    \frac12\mathbb{E}[h(Q)]. 
\]  First, since \(Q^2s'(Q)\) is even,
\begin{align}
    \mathbb{E}[s(Q)Q^2s'(Q)]
    &=
    \frac12\mathbb{E}[Q^2s'(Q)]
     =
    \frac{R^2\nu(R)}2.
    \label{eq:symmetry-first}
\end{align}
Also, \(s''\) is odd, so \(Qs''(-Q)\) is even and
\begin{align}
    \mathbb{E}[s(Q)Qs''(-Q)]
    =
    \frac12\mathbb{E}[Qs''(-Q)] =
    -\frac12\mathbb{E}[Qs''(Q)] =
    -\frac R2\mathbb{E}[Gs''(RG)] =
    \frac{\mu(R)-\nu(R)}2,
    \label{eq:symmetry-second}
\end{align}
where the last equality uses
\(\nu(R)=\mu(R)+R\mu'(R)\). To compute the Jacobian, we first differentiate \eqref{eq:prox-foc}
implicitly. Writing
\[
    D(t;u,\ell,a)
    :=
    1+t^2\ell s'\big(\mathcal X(t;u,\ell,a)\big),
\]
we obtain
\[
    \partial_u\mathcal X
    =
    \frac{tW}{2\sqrt{u}D},
    \qquad
    \partial_\ell\mathcal X
    =
    -\frac{t^2s(\mathcal X)}{D},
    \qquad
    \partial_a\mathcal X
    =
    -\frac{Q}{D}.
\]
Hence, at the base point in \eqref{eq:app-base-solution},
\[
    \partial_u\mathcal X
    =
    \partial_\ell\mathcal X
    =
    0,
    \qquad
    \partial_a\mathcal X=-Q.
\]
Using these identities and (\ref{eq:app-basic-symmetry}), direct
differentiation gives   
\begin{align*}
    \partial_uF_1&=1,
    &
    \partial_\ell F_1&=-2,
    &
    \partial_aF_1
    &=
    4\ell^2\mathbb{E}\bigl[Q\{s'(Q)\}^2\bigr]
    =0,\\
    \partial_uF_2&=0,
    &
    \partial_\ell F_2&=0,
    &
    \partial_aF_2
    &=
    -\mathbb{E}[s(Q)Q^2s'(Q)]
    =-\frac{R^2\nu(R)}2,\\
    \partial_uF_3&=0,
    &
    \partial_\ell F_3&=-\mu(R),
    &
    \partial_aF_3
    &=
    \frac{2}{\mu(R)}
    \mathbb{E}[s(Q)Qs''(-Q)]
    =
    \frac{\mu(R)-\nu(R)}{\mu(R)}.
\end{align*} Therefore,
    \begin{equation}
        D_{(u,\ell,a)}F
        \left(
            0,
            \frac{1}{\mu(R)},
            \frac{1}{\mu(R)},
            1
        \right)
        =
        \begin{pmatrix}
            1
            &
            -2
            &
            0
            \\[3pt]
            0
            &
            0
            &
            -\dfrac{R^2\nu(R)}{2}
            \\[8pt]
            0
            &
            -\mu(R)
            &
            \dfrac{\mu(R)-\nu(R)}{\mu(R)}
        \end{pmatrix}.
        \label{eq:app-state-jacobian}
    \end{equation}
    Its determinant is
    $
        -\frac{R^2\mu(R)\nu(R)}{2}\neq0.$ 
    The implicit function theorem  therefore produces a unique smooth local
    solution
    $
        t
        \longmapsto
        \bigl(
            u(t),\ell(t),a(t)
        \bigr).$

    Finally, replacing \(t\) by \(-t\) in
    \eqref{eq:app-prox-implicit}
    has the same effect in distribution as replacing \(W\) by \(-W\).
    Since \(W\) is symmetric,
    \[
        F(-t,u,\ell,a)
        =
        F(t,u,\ell,a).
    \]
    The uniqueness of the local implicit-function solution therefore implies that \(u(t)\), \(\ell(t)\), and \(a(t)\) are even functions of \(t\). Consequently, for every fixed \(k\ge0\),
    \[
        u(t)
        =
        \sum_{j=0}^{k}u_{2j}t^{2j}
        +
        O_R(t^{2k+2}),
    \]
    and analogously for \(\ell(t)\) and \(a(t)\). Equivalently, after setting \(\delta=t^2\), these functions admit finite-order expansions in integer powers of \(\delta\).

    Shrinking the neighborhood if necessary, the local solution satisfies \(u(t),\ell(t),a(t)>0\). For every sufficiently small \(t>0\), it solves \eqref{eq:app-scaled-state-1}--\eqref{eq:app-scaled-state-3} with \(\delta=t^2\). By the uniqueness of the state-equation solution in the MLE-existence region \cite[Lem. 3.1 \& discussion following Eq.~(3.4)]{zhao2022asymptotic}, this local branch coincides with \((u_\delta,\ell_\delta,a_\delta)\) for all sufficiently small \(\delta>0\).
\end{proof}

Next, we establish a more precise expression for $a_\delta$.

\begin{lem}[First-order expansion of \(a_{\delta}\)]
    \label{lem:app-m-expansion}
    For every fixed \(R>0\),
    \begin{equation}
        a_{\delta}
        =
        1
        +
        \frac{
            \mu(R)-\nu(R)
        }{
            2R^2\mu(R)\nu(R)
        }
        \delta
        +
        O_R(\delta^2).
        \label{eq:app-m-expansion}
    \end{equation}
\end{lem}

\begin{proof}
    Set \(t:=\sqrt{\delta}\), and let
    \(
        t
        \longmapsto
        \bigl(
            u(t),\ell(t),a(t)
        \bigr)
    \) denote the smooth local solution constructed in
    Lemma~\ref{lem:app-smooth-branch}, so that, for \(t\geq 0\),
    \[
        u(t)=u_{t^2},
        \qquad
        \ell(t)=\ell_{t^2},
        \qquad
        a(t)=a_{t^2}.
    \]
    Recall that
    $
        u(0)=\ell(0)=\frac{1}{\mu(R)},
        ~
        a(0)=1.$
    Moreover, the three functions are even in \(t\). Consequently,
    \begin{equation}
        u'(0)=\ell'(0)=a'(0)=0.
        \label{eq:app-branch-first-derivatives-zero}
    \end{equation}
    Define
    \[
        \mathcal{X}(t)
        :=
        \operatorname{prox}_{t^2\ell(t)\rho}
        \big(
            -a(t)Q+t\sqrt{u(t)}\,W
        \big).
    \]
    By the first-order condition for the proximal operator,
    \(\mathcal{X}(t)\) is the unique solution of
    \begin{equation}
        \mathcal{X}(t)
        +
        t^2\ell(t)s\bigl(\mathcal{X}(t)\bigr)
        =
        -a(t)Q+t\sqrt{u(t)}\,W.
        \label{eq:app-prox-branch-equation}
    \end{equation}
    All differentiations below are justified by the smoothness and
    Gaussian-integrability argument in
    Lemma~\ref{lem:app-smooth-branch}. 

    Setting \(t=0\) in
    \eqref{eq:app-prox-branch-equation} gives
    \begin{equation}
        \mathcal{X}(0)=-Q.
        \label{eq:app-X-zero}
    \end{equation}
    Differentiating \eqref{eq:app-prox-branch-equation} once with respect
    to \(t\), and then evaluating at \(t=0\), yields
    \begin{equation}
        \mathcal{X}'(0)
        =
        \sqrt{u(0)}\,W
        =
        \frac{W}{\sqrt{\mu(R)}}.
        \label{eq:app-X-first-derivative}
    \end{equation}
    Differentiating a second time, and using
    \eqref{eq:app-branch-first-derivatives-zero}, gives
    \begin{equation}
        \mathcal{X}''(0)
        +
        2\ell(0)s\bigl(\mathcal{X}(0)\bigr)
        =
        -a''(0)Q.
    \end{equation}
    Note that in the second differentiation, all terms involving
\(u'(0)\), \(\ell'(0)\), and \(a'(0)\) vanish by \eqref{eq:app-branch-first-derivatives-zero}; moreover,
\[
    \left.\frac{d^2}{dt^2}
    \bigl\{t\sqrt{u(t)}\bigr\}\right|_{t=0}
    =
    \frac{u'(0)}{\sqrt{u(0)}}=0.
\]
    Hence, by \eqref{eq:app-X-zero} and
    \(\ell(0)=\mu(R)^{-1}\),
    \begin{equation}
        \mathcal{X}''(0)
        =
        -a''(0)Q
        -
        \frac{2}{\mu(R)}s(-Q).
        \label{eq:app-X-second-derivative}
    \end{equation}

    We now use the second state equation
    \eqref{eq:app-scaled-state-2}. Along the smooth solution branch, it
    takes the form
    \begin{equation}
        \mathbb{E}\left[
            s(Q)Q s\bigl(\mathcal{X}(t)\bigr)
        \right]
        =
        0
        \label{eq:app-second-state-along-branch}
    \end{equation}
    for every \(t\) sufficiently close to zero. Differentiating
    \eqref{eq:app-second-state-along-branch} twice with respect to \(t\)
    and evaluating at \(t=0\), we obtain
    \begin{align}
        0
        &=
        \mathbb{E}\left[
            s(Q)Q
            \left\{
                s''(-Q)\bigl(\mathcal{X}'(0)\bigr)^2
                +
                s'(-Q)\mathcal{X}''(0)
            \right\}
        \right].
        \label{eq:app-second-state-second-derivative}
    \end{align}
    Substituting \eqref{eq:app-X-first-derivative} and
    \eqref{eq:app-X-second-derivative}, using
    \(s'(-Q)=s'(Q)\), and then taking expectation over the independent
    standard normal variable \(W\), gives
    \begin{align}
        0
        ={}&
        \frac{1}{\mu(R)}
        \mathbb{E}\left[
            s(Q)Q s''(-Q)
        \right]
        -
        a''(0)
        \mathbb{E}\left[
            s(Q)Q^2s'(Q)
        \right] 
        -
        \frac{2}{\mu(R)}
        \mathbb{E}\left[
            s(Q)Q s'(Q)s(-Q)
        \right].
        \label{eq:app-second-state-second-derivative-expanded}
    \end{align}
    The last expectation vanishes. Indeed,
    \(s(Q)s(-Q)=s'(Q)\), and hence
    \(
        \mathbb{E}\left[
            s(Q)Q s'(Q)s(-Q)
        \right]
        =
        \mathbb{E}\left[
            Q\{s'(Q)\}^2
        \right]
        =
        0
    \)
    by oddness. Furthermore,
    (\ref{eq:symmetry-first}) and
    (\ref{eq:symmetry-second}) give
    \[
        \mathbb{E}\left[
            s(Q)Q^2s'(Q)
        \right]
        =
        \frac{R^2\nu(R)}{2},
        \qquad
        \mathbb{E}\left[
            s(Q)Q s''(-Q)
        \right]
        =
        \frac{\mu(R)-\nu(R)}{2}.
    \]
    Therefore,
    \eqref{eq:app-second-state-second-derivative-expanded} reduces to
    \[
        0
        =
        \frac{\mu(R)-\nu(R)}{2\mu(R)}
        -
        a''(0)\frac{R^2\nu(R)}{2}.
    \]
    Solving for \(a''(0)\), we obtain
    \begin{equation}
        a''(0)
        =
        \frac{
            \mu(R)-\nu(R)
        }{
            R^2\mu(R)\nu(R)
        }.
        \label{eq:app-m-second-derivative}
    \end{equation}

    Finally, since \(a(t)\) is smooth and even, Taylor's theorem gives
    \(
        a(t)
        =
        a(0)
        +
        \frac{1}{2}a''(0)t^2
        +
        O_R(t^4).\)
    Recalling that \(a(0)=1\), \(t^2=\delta\), and
    \(a_{\delta}=a(\sqrt{\delta})\), we conclude from
    \eqref{eq:app-m-second-derivative} that
    \[
        a_{\delta}
        =
        1
        +
        \frac{
            \mu(R)-\nu(R)
        }{
            2R^2\mu(R)\nu(R)
        }
        \delta
        +
        O_R(\delta^2),
    \]
    which proves \eqref{eq:app-m-expansion}.
\end{proof} 

 We are now ready to establish Proposition \ref{prop:asymptotic-radial-bias}.

\begin{proof}[Proof of Proposition \ref{prop:asymptotic-radial-bias}]
    By Lemmas~\ref{lem:app-smooth-branch} and
    \ref{lem:app-m-expansion}, we have
    \[
        a_{\delta}
        =
        1
        +
        \frac{
            \mu(R)-\nu(R)
        }{
            2R^2\mu(R)\nu(R)
        }
        \delta
        +
        O_R(\delta^2),\qquad\tau_{\delta}^2
        =
        u_{\delta}
        =
        \frac{1}{\mu(R)}
        +
        O_R(\delta).
    \]
    Therefore,
    \begin{align*}
        r_{\mathrm{MLE}}(\delta,R)^2
        &=
        a_{\delta}^2R^2
        +
        \delta\tau_{\delta}^2
        \\
        &=
        R^2
        +
        \delta
        \left\{
            \frac{
                \mu(R)-\nu(R)
            }{
                \mu(R)\nu(R)
            }
            +
            \frac{1}{\mu(R)}
        \right\}
        +
        O_R(\delta^2)
        =
        R^2
        +
        \frac{\delta}{\nu(R)}
        +
        O_R(\delta^2).
    \end{align*}
    Since \(R>0\) is fixed, as \(x\to0\) we can write 
    $
        \sqrt{R^2+x}
        =
        R+\frac{x}{2R}+O_R(x^2).$  Hence
    \[
        r_{\mathrm{MLE}}(\delta,R)
        =
        R
        +
        \frac{\delta}{2R\nu(R)}
        +
        O_R(\delta^2).
    \]
    Recalling that \(\nu(R)=q'(R)\) proves the result.
\end{proof} 
        
\section{Technical Lemmas}
\begin{lem}[Bounding $\ell_2$ error by direction error and norm error] \label{lem:dirplusnorm}
    Under the convention $\frac{\bm{0}}{0}=\be_1$, it holds for any $\bu,\bv\in \mathbb{R}^d$ that
    \[
        \|\bu-\bv\|_2 \le \min\{\|\bu\|_2,\|\bv\|_2\}\bigg\|\frac{\bu}{\|\bu\|_2}-\frac{\bv}{\|\bv\|_2}\bigg\|_2+\big|\|\bu\|_2-\|\bv\|_2\big|.\] 
\end{lem}
\begin{proof}
We first treat $\bu,\bv\ne \bm{0}$. By triangle inequality,
\begin{align}\label{uupper1}
    \|\bu-\bv\|_2 &\le \bigg\|\bu-\frac{\bv\|\bu\|_2}{\|\bv\|_2}\bigg\|_2 + \bigg\|\frac{\bv\|\bu\|_2}{\|\bv\|_2}-\bv\bigg\|_2 = \|\bu\|_2\bigg\|\frac{\bu}{\|\bu\|_2}-\frac{\bv}{\|\bv\|_2}\bigg\|_2 + \big|\|\bu\|_2-\|\bv\|_2\big|. 
\end{align}
Swapping the role of $\bu,\bv$ yields
\begin{align}\label{vupper2}
     \|\bu-\bv\|_2 \le\|\bv\|_2\bigg\|\frac{\bu}{\|\bu\|_2}-\frac{\bv}{\|\bv\|_2}\bigg\|_2 + \big|\|\bu\|_2-\|\bv\|_2\big|. 
\end{align}
Combining (\ref{uupper1})--(\ref{vupper2}) yields the claimed inequality. It remains to check the case of $\bu=\bm{0}$ or $\bv=\bm{0}$, under the convention $\frac{\bm{0}}{0}=\be_1$. In this case, we have $\|\bu-\bv\|_2 = |\|\bu\|_2-\|\bv\|_2|$, hence the claim follows since $\min\{\|\bu\|_2,\|\bv\|_2\}\|\frac{\bu}{\|\bu\|_2}-\frac{\bv}{\|\bv\|_2}\|_2$ is always non-negative. 
\end{proof}
\begin{lem}[Lower bounding $\ell_2$ error by direction error] \label{lem:lowerl2direction}
    Under the convention $\frac{\bm{0}}{0}=\be_1$, it holds for any $\bu,\bv\in \mathbb{R}^d$ that
    \[
        \|\bu-\bv\|_2\ge \frac{1}{2}\max\{\|\bu\|_2,\|\bv\|_2\}\bigg\|\frac{\bu}{\|\bu\|_2}-\frac{\bv}{\|\bv\|_2}\bigg\|_2. 
    \]
\end{lem}
\begin{proof}
    We first treat $\bu,\bv\ne \bm{0}$, then a standard bound (e.g., \cite[Fact 13]{matsumoto2025learning}) gives 
    \[
        \bigg\|\frac{\bu}{\|\bu\|_2}-\frac{\bv}{\|\bv\|_2}\bigg\|_2\le \frac{2\|\bu-\bv\|_2}{\max\{\|\bu\|_2,\|\bv\|_2\}}.
    \]
    Rearranging yields the claimed inequality. It remains to check $\bu=\bm{0}$ or $\bv=\bm{0}$, under the convention $\frac{\bm{0}}{0}=\be_1$. The claimed inequality holds trivially if $\bu=\bv=\bm{0}$. We now validate the case of $\bu\ne \bm{0}$, $\bv=\bm{0}$, under which we have 
    \[
       \frac{1}{2}\max\{\|\bu\|_2,\|\bv\|_2\}\bigg\|\frac{\bu}{\|\bu\|_2}-\frac{\bv}{\|\bv\|_2}\bigg\|_2=\frac{1}{2}\|\bu\|_2\bigg\|\frac{\bu}{\|\bu\|_2}-\be_1\bigg\|_2\le \|\bu\|_2=\|\bu-\bv\|_2,
    \]
    again yielding the desired inequality. The proof is complete. 
\end{proof} 
\begin{lem}[Standard bound on maximum of Gaussian variables]\label{lem:maxgaussian}
    If $g_1,...,g_n$ are $N(0,1)$ variables, then with probability at least $1-2n^{-1}$, 
    \[\max_{1\le i\le n}|g_i|\le 2\sqrt{\log n}.\]
\end{lem}
\begin{proof}
    By standard Gaussian tail bound,  $\mathbb{P}(|g_i|\ge t)\le 2\exp(-\frac{t^2}{2})$ holds for any $i\in[n]$. Taking a union bound over $i$ yields \[\mathbb{P}\bigg(\max_{1\le i\le n}|g_i|\ge t\bigg)\le 2n\exp\Big(-\frac{t^2}{2}\Big).\] Setting $t=2\sqrt{\log n}$ yields the claim. 
\end{proof}
\begin{lem}
    \label{lem:operconcen}
    Let $\bx_1,...,\bx_n$ be i.i.d. $N(\bm{0},\bI_d)$ vectors. Under $n\ge d$, there exists universal constant $C$ such that 
    \[\mathbb{P}\bigg(\bigg\|\frac{1}{n}\sum_{i=1}^n\bx_i\bx_i^\top\bigg\|_{\rm op}\le C,~\frac{1}{n}\sum_{i=1}^n\|\bx_i\|_2\le C\sqrt{d}\bigg)\ge 1-2e^{-d}.\]
\end{lem}
\begin{proof}
   Let $\bX^\top=[\bx_1,\bx_2,...,\bx_n]$, which is a $d\times n$ matrix of i.i.d. $N(0,1)$ variables. Then \[\bigg\|\frac{1}{n}\sum_{i=1}^n\bx_i\bx_i^\top\bigg\|_{\rm op} = \bigg\|\frac{1}{n}\bX^\top\bX\bigg\|_{\rm op}=\frac{1}{n}\|\bX\|_{\rm op}^2.\] 
   By \cite[Theorem 4.4.5]{vershynin2018high}, under $n\ge d$,  $\|\bX\|_{\rm op}\le C\sqrt{n}$ holds with probability at least $1-2e^{-d}$. This directly implies the desired $\|\frac{1}{n}\sum_{i=1}^n \bx_i\bx_i^\top\|_{\rm op} =O(1)$. We now show that this also implies the desired bound on $\frac{1}{n}\sum_{i=1}^n\|\bx_i\|_2$. In fact, 
   \[\frac{1}{n}\sum_{i=1}^n\|\bx_i\|_2\le \bigg(\frac{1}{n}\sum_{i=1}^n\|\bx_i\|_2^2\bigg)^{1/2} = \frac{\|\bX\|_{\rm F}}{\sqrt{n}}\le \frac{\sqrt{d}\|\bX\|_{\rm op}}{\sqrt{n}}\le C\sqrt{d}.\] The proof is complete. 
\end{proof}

\begin{lem}[Population gradient of the logistic loss]
\label{lem:population-gradient}
For every \(r>0\) and every \(\bv\in\mathbb S^{d-1}\), the gradient of the population loss $L(\btheta)=\mathbb{E}L_n(\btheta)$ is given by 
\[\nabla L(r\bv)
    =
    q(r)\bv-q(R)\bv^*.\]
\end{lem}

\begin{proof}
By exchanging the expectation and differentiation and the definition of the logistic model, 
\begin{align}
    \nabla L(r\bv) = \nabla \mathbb{E}L_n(\btheta) = \mathbb{E}\nabla L_n(\btheta) = \mathbb{E}\bigg[\frac{1}{n}\sum_{i=1}^n\big(s(\bx_i^\top\btheta)-y_i\big)\bx_i\bigg]=\mathbb{E}\big[\big(s(\bx_i^\top\btheta)-s(\bx_i^\top\btheta^*)\big)\bx_i\big].\label{exchan}
\end{align}
Since \(\bx_i\sim N(\bm{0},\bI_d)\),  for any $\bu\in\mathbb{R}^d\setminus \{\bm{0}\}$, rotational invariance gives 
\[
    \mathbb{E}[s(\bx_i^\top\bu)\bx_i] = \mathbb{E}\bigg[s(\bx_i^\top\bu)\bx_i^\top\frac{\bu}{\|\bu\|_2}\bigg]\frac{\bu}{\|\bu\|_2} = \mathbb{E}_{g\sim N(0,1)}\bigg[s(\|\bu\|_2g)g\bigg] \frac{\bu}{\|\bu\|_2} = q(\|\bu\|_2)\frac{\bu}{\|\bu\|_2}. 
\]
Applying this formula to (\ref{exchan}) yields the desired claim. 
\end{proof}
The following lemma identifies the order of $q'(r)= \mathbb{E}_{g\sim N(0,1)}[g^2s'(rg)]$. 
\begin{lem}[Order of $q'(r)$ {\cite[Lemma 15]{chen2026finite}}]
\label{lem:qprime}
We have that $q'(r)\asymp (1+r)^{-3}$, i.e., there exist universal constants $c$ and $C$ such that 
\[
    \frac{c}{(1+r)^3}\le q'(r)\le\frac{C}{(1+r)^3},\quad \forall r>0.  
\]
\end{lem}

\begin{lem}[Bernstein's inequality {\cite[Theorem 2.10]{13concen}}]
    \label{lem:bernstein210} Let $X_1,...,X_n$ be independent random variables, and assume that for some $v,c>0$, $  \sum_{i=1}^n \mathbb{E}X_i^2 \le v$ and $\sum_{i=1}^n \mathbb{E}|X_i|^q\le \frac{q!}{2}vc^{q-2}~~(\textrm{for all integers }q\ge 3)$
    hold, then we have
    \[
        \mathbb{P}\bigg(\bigg|\sum_{i=1}^n(X_i-\mathbb{E}X_i)\bigg|\ge \sqrt{2vt}+ct\bigg) \le 2\exp(-t),\quad \forall t>0.\]
\end{lem}

\begin{lem}[Hanson-Wright inequality {\cite[Theorem 6.2.1]{vershynin2018high}}]
    \label{lem:hanson-wright}
    Let $\bz\sim N(0,\bI_N)$ and $\bA\in\mathbb{R}^{N\times N} $ be deterministic and symmetric. For some universal constant $C>0$, it holds for any $t\ge 0$ that
    \[\mathbb{P}\bigg(\Big|\bz^\top\bA\bz-\tr(\bA)\Big|\le C\big(\|\bA\|_{\rm F}\sqrt{t}+\|\bA\|_{\rm op}t\big)\bigg)\ge 1-2e^{-t}.\]
\end{lem}
\begin{lem}[Pinsker's inequality {\cite[Lemma 15.2]{wainwright2019high}}]
\label{lem:pinsker}
Let $P$ and $Q$ be two probability measures on the same measurable
space. Define their total variation distance by
\(\|P-Q\|_{\rm TV}
    :=
    \sup_{A}|P(A)-Q(A)|.\)
Then
\[
    \|P-Q\|_{\rm TV}
    \le
    \sqrt{\frac{1}{2}D_{\rm KL}(P\|Q)},
\]
where
\(
    D_{\rm KL}(P\|Q)
    :=
    \int \log\big(\frac{{\rm d}P}{{\rm d}Q}\big){\rm d}P 
\) denotes the KL-divergence, 
with the convention that $D_{\rm KL}(P\|Q)=+\infty$ if
$P$ is not absolutely continuous with respect to $Q$.
\end{lem}

\begin{lem}[Le Cam's  inequality
{\cite[Corollary 15.6]{wainwright2019high}}]
\label{lem:lecam-two-point}
Let $P_0$ and $P_1$ be two probability measures, and let
$\vartheta(P)\in\mathbb{R}$ be a scalar-valued functional.
Then
\[
    \inf_{\hat{\vartheta}}
    \max_{j\in\{0,1\}}
    \mathbb{E}_{P_j}
    \big|
        \hat{\vartheta}-\vartheta(P_j)
    \big|
    \ge
    \frac{
        |\vartheta(P_1)-\vartheta(P_0)|
    }{4}
    \left(
        1-\|P_0-P_1\|_{\rm TV}
    \right),
\]
where the infimum is taken over all estimators
$\hat{\vartheta}$ measurable with respect to the observation.
\end{lem}
\end{document}

%% file: preamble.tex
\usepackage{amssymb,amsmath,amsfonts,latexsym}
\usepackage{amsmath,graphicx,bm,xcolor,url}
\usepackage[caption=false]{subfig} 
\usepackage{array}
\usepackage{verbatim}
\usepackage{bm}
\usepackage{verbatim}
\usepackage{textcomp}
\usepackage{mathrsfs}
\usepackage{relsize}
\usepackage{subfig}
 \usepackage{amsthm}

\catcode`~=11 \def\UrlSpecials{\do\~{\kern -.15em\lower .7ex\hbox{~}\kern .04em}} \catcode`~=13 

\allowdisplaybreaks[3]

\newcommand{\norm}[1]{\left\Vert#1\right\Vert}

\newcommand{\nn}{\nonumber}

\newcommand{\calI}{\mathcal{I}}

\newcommand{\calV}{\mathcal{V}}

\newcommand{\ba}{\mathbf{a}}
\newcommand{\bA}{\mathbf{A}}

\newcommand{\bB}{\mathbf{B}}

\newcommand{\be}{\mathbf{e}}

\newcommand{\bg}{\mathbf{g}}

\newcommand{\bI}{\mathbf{I}}

\newcommand{\bP}{\mathbf{P}}

\newcommand{\bu}{\mathbf{u}}

\newcommand{\bv}{\mathbf{v}}

\newcommand{\bw}{\mathbf{w}}

\newcommand{\bx}{\mathbf{x}}
\newcommand{\bX}{\mathbf{X}}

\newcommand{\bz}{\mathbf{z}}
\newcommand{\bZ}{\mathbf{Z}}

\newcommand{\btheta}{\bm{\theta}}

\DeclareMathOperator{\diag}{diag}

\DeclareMathOperator{\tr}{tr}

\newtheorem{theorem}{Theorem}

\newtheorem{assumption}{Assumption}
\newtheorem{corollary}{Corollary}

\newcommand{\qednew}{\nobreak \ifvmode \relax \else
      \ifdim\lastskip<1.5em \hskip-\lastskip
      \hskip1.5em plus0em minus0.5em \fi \nobreak
      \vrule height0.75em width0.5em depth0.25em\fi}



%% file: libr.bib
@article{sur2019modern,
  title={A modern maximum-likelihood theory for high-dimensional logistic regression},
  author={Sur, Pragya and Cand{\`e}s, Emmanuel J},
  journal={Proceedings of the National Academy of Sciences of the United States of America},
  volume={116},
  number={29},
  pages={14516--14525},
  year={2019}
}

@article{candes2020phase,
  title={The phase transition for the existence of the maximum likelihood estimate in high-dimensional logistic regression},
  author={Cand{\`e}s, Emmanuel J and Sur, Pragya},
  journal={The Annals of Statistics},
  volume={48},
  number={1},
  pages={27--42},
  year={2020},
  publisher={JSTOR}
}

@article{kuchelmeister2024finite,
  title={Finite sample rates for logistic regression with small noise or few samples},
  author={Kuchelmeister, Felix and van de Geer, Sara},
  journal={Sankhya A},
  pages={1--70},
  year={2024},
  publisher={Springer}
}

@article{chardon2024finite,
  title={Finite-sample performance of the maximum likelihood estimator in logistic regression},
  author={Chardon, Hugo and Lerasle, Matthieu and Mourtada, Jaouad},
  journal={arXiv preprint arXiv:2411.02137},
  year={2024}
}

@book{13concen,
    author = {Boucheron, Stéphane and Lugosi, Gábor and Massart, Pascal},
    title = {Concentration Inequalities: A Nonasymptotic Theory of Independence},
    publisher = {Oxford University Press},
    year = {2013},
    month = {February}
}

@inproceedings{hsu2024sample,
  title={On the sample complexity of parameter estimation in logistic regression with normal design},
  author={Hsu, Daniel and Mazumdar, Arya},
  booktitle={The Thirty Seventh Annual Conference on Learning Theory},
  pages={2418--2437},
  year={2024},
  organization={PMLR}
}

@article{zhao2022asymptotic,
  title={The asymptotic distribution of the MLE in high-dimensional logistic models: Arbitrary covariance},
  author={Zhao, Qian and Sur, Pragya and Candes, Emmanuel J},
  journal={Bernoulli},
  volume={28},
  number={3},
  pages={1835--1861},
  year={2022},
  publisher={Bernoulli Society for Mathematical Statistics and Probability}
}

@article{ostrovskii2021finite,
  author  = {Ostrovskii, Dmitrii M. and Bach, Francis},
  title   = {Finite-Sample Analysis of {$M$}-Estimators Using Self-Concordance},
  journal = {Electronic Journal of Statistics},
  volume  = {15},
  number  = {1},
  pages   = {326--391},
  year    = {2021}
}

@article{chinot2020robust,
  title={Robust statistical learning with Lipschitz and convex loss functions},
  author={Chinot, Geoffrey and Lecu{\'e}, Guillaume and Lerasle, Matthieu},
  journal={Probability Theory and related fields},
  volume={176},
  number={3},
  pages={897--940},
  year={2020},
  publisher={Springer}
}

@article{chen2026finite,
  title={Finite-Sample Performance of Gradient Descent in Logistic Regression with Gaussian Design},
  author={Chen, Junren and Mazumdar, Arya},
  journal={arXiv preprint arXiv:2606.21683},
  year={2026}
}

@InProceedings{matsumoto2025learning,
  title = 	 {Learning sparse generalized linear models with binary outcomes via iterative hard thresholding},
  author =       {Matsumoto, Namiko and Mazumdar, Arya},
  booktitle = 	 {Proceedings of Thirty Eighth Conference on Learning Theory},
  pages = 	 {3933--4032},
  year = 	 {2025},
  volume = 	 {291},
  series = 	 {Proceedings of Machine Learning Research},
  month = 	 {30 Jun--04 Jul},
  publisher =    {PMLR}
}

@book{wainwright2019high,
  author    = {Wainwright, Martin J.},
  title     = {High-Dimensional Statistics: A Non-Asymptotic Viewpoint},
  publisher = {Cambridge University Press},
  year      = {2019}
}

@book{vershynin2018high,
  title={High-dimensional probability: An introduction with applications in data science},
  author={Vershynin, Roman},
  volume={47},
  year={2018},
  publisher={Cambridge university press}
}
